\documentclass[10pt]{amsart}

\usepackage{amsmath,amssymb}

\usepackage[active]{srcltx}
\usepackage{graphicx,color}

\usepackage{epic}
\usepackage{pstricks}
\usepackage{amsthm}

\newtheorem{theorem}{Theorem}[section]
\newtheorem{proposition}[theorem]{Proposition}

\newtheorem{lemma}[theorem]{Lemma}

\newtheorem{corollary}[theorem]{Corollary}
\theoremstyle{definition}

\newtheorem{remark}[theorem]{Remark}

\def\R{\mathbb{R}}

\def\N{\mathbb{N}}
\def\C{\mathbb{C}}

\def\E{\mathbb{E}}

\def \sgn{\operatorname{sgn}}

\def \supp{\operatorname{supp}}

\def\cA{{\mathcal A}}
\def\cC{{\mathcal C}}
\def\cD{{\mathcal D}}

\def\cF{{\mathcal F}}

\def\cK{{\mathcal K}}
\def\cL{{\mathcal L}}

\def\cR{{\mathcal R}}
\def\cS{{\mathcal S}}

\def\cW{{\mathcal W}}
\def\cE{{\mathcal E}}
\def\cO{{\mathcal O}}

\def\cG{{\mathcal G}}
\def\cT{{\mathcal T}}

\def\cB{{\mathcal B}}
\def\cP{{\mathcal P}}

\def\cJ{{\mathcal J}}

\def\m{\mathfrak{M}}

\newcommand{\ud}{\mathrm{d}}

\newcommand{\X}{{X}}

\newcommand{\B}{{B}}

\def\trXY1-tq{Tr(X,Y, 1-\theta ,q)}

\newcommand{\bracket}[2]{\left\langle #1 , #2 \right\rangle}

\definecolor{darkred}{rgb}{0.7,0.1,0.1}

\title[The maximal regularity of integro-differential equations]{The maximal regularity property of abstract integro-differential equations}

\author{Sebastian Kr\'ol}
\address{Sebastian Kr{\'o}l, Faculty of Mathematics and Computer Science, Adam Mickiewicz University Pozna{\'n}, ul. Uniwersytetu Pozna{\'n}skiego 4, 61-614 Pozna{\'n}, Poland}
\email{sebastian.krol@amu.edu.pl}

\thanks{}

\begin{document}

\keywords{Besov spaces, Triebel-Lizorkin spaces,  integro-differential equations, maximal regularity, Fourier multipliers, Hardy-Littlewood maximal operator, Rubio de Francia iteration algorithm}

\subjclass{42B37, 42B20, 42A45, 45N05, 46N20}  

\begin{abstract}
We provide a convenient framework for the study of the well-posedness of a variety of abstract (integro)differential equations in general Banach function spaces. It allows us to extend and complement the known theory on the maximal regularity of such equations. 

More precisely, by methods of harmonic analysis, we identify large classes of Banach spaces which are invariant with respect to distributional Fourier multipliers. 
Such classes include general vector-valued Banach function spaces $\Phi$ and/or the scales of Besov and Triebel-Lizorkin spaces defined by $\Phi$.

We apply this result to the study of the well-posedness and maximal regularity property of abstract second-order integro-differential equation, which models various types of elliptic and parabolic problems arising in different areas of applied mathematics. 
\end{abstract}

\renewcommand{\subjclassname}{\textup{2010} Mathematics Subject Classification}

\maketitle

\section{Introduction with preliminary considerations}\label{intro}

The paper continues the well-known research program on the study of the well-posedness and maximal regularity of pseudodifferential equations in a variety of function spaces via techniques of vector-valued Fourier multiplier theory.  
In a quite general framework, this program was already formulated by Amann in \cite[Chapter 3]{Am95} and was labelled as '{\it pairs of maximal regularity}'. In one of the pioneering papers in this area, \cite{Am97}, Amann showed how the boundedness of translation-invariant operators with operator-valued symbols guarantees the well-posedness of various elliptic and parabolic (integro)differential equations in a scale of vector-valued Besov spaces. His paper constitutes a conceptual background and toolbox for further extensions in the literature.
Here, we just mention the breakthrough result due to Weis \cite{We01} on the characterization of the $L^p$-maximal regularity of the first-order Cauchy problems and a few others, seminal papers \cite{ArBu02}, \cite{ArBaBu04}, \cite{ArBu04}, \cite{ChSr05}, \cite{ChFi14}, \cite{KeLi04} (see also the references therein).
In this paper, we provide a convenient framework for such studies. In particular, it allows us to extend and unify several results from the literature.

 To describe our main results and explain difficulties arising during their studies, for the transparency, we consider only a special form of a more general abstract second-order integrodifferential equation \eqref{P} discussed in Section \ref{app}; see Theorem \ref{main for itegro-dff} with Corollary \ref{for umd} for the corresponding results. 
Moreover, to keep this presentation compact, we closely follow the notation and terminology from Amann's paper \cite{Am97}, which we only roughly recall below, and refer the reader to \cite{Am97} for more details in the case.  It allows us to keep the novelty of our results more transparent. 

Consider the following abstract evolution problem on the line:
\begin{equation}\label{P1}
 \partial u + A u  = f \quad \textrm{ in } \cS'(\R;X).
\end{equation}
Here $X$ is a Banach space and by $\partial$ we denote a distributional derivative, that is, $\partial u (\phi) := - u(\partial \phi)$ for every $X$-valued distribution $u\in \cD'(X):=\cD'(\R; X)$ and every test function $\phi \in \cD:=\cD(\R; \C)$. Moreover, $A$ stands for a closed, linear operator on  $X$ with domain $D_A$ equipped with the graph norm and denoted by $Y$ in this section.

Under suitable assumptions on the operator $A$ one can show that for every tempered distribution $f\in \cS'(X):=\cS'(\R; X)$ the problem \eqref{P1} has a unique distributional solution $u_f\in \cS'(Y)$. For instance, suppose that $i\R\subset \rho(A)$, where $\rho(A)$ stands for the resolvent set of $A$ and that the function  $a(t):=(it + A)^{-1} \in \cL(X,Y)$, $t\in \R$, has a polynomial growth at infinity.
Then, one can easily argue that the {\it solution operator} $U$ for \eqref{P1} is given by 
\begin{equation}
U: \cS'(X)\ni f\mapsto  \cF^{-1} \Theta (a, \cF f) \in \cS'(Y) \subset \cS'(X).
\end{equation}
Here, $\Theta$ denotes a hypocontinuous, bilinear map from $O_M(\cL(X,Y))\times \cS'(X)$ into $\cS'(Y)$ (see the kernel theorem, e.g. in \cite[Theorem 2.1]{Am97}), and  $O_M(\cL(X,Y))$ stands for the space of slowly increasing, smooth $\cL(X,Y)$-valued functions on $\R$.

One of the main questions of the distributional theory of partial differential equations, in the context of \eqref{P1}, is to specify further assumptions on $A$ and $f$ such that \eqref{P1} is satisfied in a more classical, strong sense.
That is, naturally restricting $f$ to the class of regular distributions, i.e. $f \in \cS'_r(X):= L^1_{loc}(X)\cap \cS'(X)$, we ask when $Uf$ is in the Sobolev class $W^{1,1}_{loc}(X)$, which means that $Uf$, $\partial Uf \in L_{loc}^1(X)$. Then,  in particular, the classical derivative $(Uf)'$ of $Uf$ exists a.e. on $\R$, $\partial Uf = (Uf)'$, $Uf(t)\in Y$ for a.e. $t\in \R$,  and \eqref{P1} is satisfied pointwise a.e on $\R$.
In other words, we would like to identify a subspace of $\cS'_r(X)$ such that the solution operator $U$ maps it into $L^1_{loc}(Y)$ or equivalently
the operator $AU$ maps it into $L^1_{loc}(X)$.

The other related question, which is of main importance in the study of nonlinear problems associated with $A$, concerns the property of {\it maximal regularity with respect to a given Banach space} $\E(X) := \E(\R;X)\subset \cS'(X)$.
 That is, under suitable assumptions on $A$ and $X$, we would like to identify a class of Banach spaces $\E(X)\subset \cS'(X)$ such that for all $f\in \E(X)$ each summand on the left side of \eqref{P1}, i.e. $\partial Uf$ and $AUf$, belongs to $\E(X)$. 
Since $0\in\rho(A)$,  it is equivalent to say that the corresponding distributional solution $Uf$ of \eqref{P1} is in $\E(Y)$ (or, equivalently,  $\partial Uf  \in \E(X)$). Therefore,  this question can be rephrased as the question about the identification of those Banach spaces $\E(X)\subset\cS'(X)$, which are {\it invariant} with respect to $AU$. 
Moreover, by the simple algebraic structure of \eqref{P1}, if (in addition) $\E(X)\subset L^1_{loc}(X)$, then  we get that $Uf$ is a strong solution of \eqref{P1}, that is, $Uf \in W^{1,1}_{loc}(X)$ for every $f\in \E(X)$.

The above questions constitute a natural research program, which can be easily reformulated for a variety of different types of (non)autonomous integrodifferential equations.  In this paper, following \cite{KeAp20, ApKe20}, we address such questions to an abstract integrodifferential equation (see \eqref{P}), which covers various types of particular abstract evolution problems already studied in the literature, e.g. the problem \eqref{P1}; see also Section \ref{cases}.

Our main result, Theorem \ref{main for itegro-dff} with Corollary \ref{for umd}, in the particular case of the equation \eqref{P1}, reads as follows:

 {\it Assume that $A$ is an invertible, bisectorial operator on a Banach space $X$. Let $\Phi$ be an arbitrary Banach function space $\Phi$ over $(\R , \ud t)$ such that the Hardy-Littlewood maximal operator $M$ is bounded on $\Phi$ and its dual $\Phi'$.
Let $\Phi(X)$ denote  the $X$-valued version of $\Phi$,  and  $B^{s,q}_\Phi(X)$ and  $F^{s,q}_\Phi(X)$ ($s\in \R, q\in [1, \infty]$) be the corresponding scales of Besov and Triebel-Lizorkin spaces defined on the basis of $\Phi$. 
Then,
\begin{itemize}
\item [(i)] if $X$ is a Hilbert space then for each 
\[
\E \in \left\{ \Phi,\, B^{s,q}_\Phi,\, F^{s,r}_\Phi : s\in \R, q\in [1,\infty], r\in (1,\infty)\right\},
\] the Banach space $\E(X)$ is invariant with respect to $AU$. In particular, 
\[
\| Uf\|_{\E(X)} + \|A Uf\|_{\E(X)} + \|\partial Uf\|_{\E(X)} \lesssim \|f\|_{\E(X)} \quad (f\in \E(X)).
\]
 Consequently, for each $f\in \Phi(X)$ the distributional solution $Uf$ of \eqref{P1} is a strong one, that is, $U(\Phi(X))\subset W^{1,1}_{loc}(X)$ and \eqref{P1} holds a.e. on $\R$.
\item [(ii)] if $X$ has $U\!M\!D$ property and, additionally, $A$ is $R$-bisectorial, then the conclusion of $(i)$ holds.
\item [(iii)] for every $s\in \R$ and $q\in [1,\infty]$, the space $B^{s,q}_{\Phi}(X)$ is invariant with respect to $AU$.\\[-1ex]
\end{itemize}}

 We refer to Section \ref{spaces} for the definition of spaces involved in the above formulation. In the context of our first question, we should also point out a consequence of Theorem \ref{RdeF}(i), which says that
\[
\bigcup \Phi (X) = \bigcup_{w\in A_2} L^2(\R; w\ud t; X).
\]%_{M \textrm{ bounded on }\Phi, \Phi'} 
The sum on the left side is taken over all Banach function spaces $\Phi$ such that the Hardy-Littlewood maximal operator $M$ is bounded on $\Phi$ and its dual $\Phi'$, and $A_2$ stands for the Muckenhoupt class of $A_2$-weights on $\R$.

Note that the formulation of the program of 'pairs of maximal regularity' from \cite[Chapter III]{Am95} slightly differs from the one expressed by these two above questions. 
Here, we do not a priori restrict ourselves to the class of spaces such that $\E(Y)\subset W^{1,1}_{loc}(X)$ (cf. \cite[Chapter III.1.5, p.94]{Am95}). 
For the problem \eqref{P1}, as was already pointed out above, if $\E(X)$ is invariant with respect to $AU$ and $\E(X)\subset \cS'_r(X)$, then a simply algebraic structure of \eqref{P1} yields immediately that $U(\E(X))\subset  W^{1,1}_{loc}(X)$. But for equations with a more general structure than that one of \eqref{P1}, the existence of distributional/strong solutions and {\it maximal regularity property} usually can be obtained under different conditions. 

In any case, the basic idea comes from the study of \eqref{P1}. It relies on the study invariant subspaces with respect to operators corresponding to the solution operator for a given problem via their boundedness properties, i.e.,  norm estimates they satisfy. 
More precisely, in the context of \eqref{P1},  the operator $U$ can be expressed as a convolution operator with a distributional kernel $\cF^{-1}a \in \cS'(\cL(X,Y))$, where $\bracket{\cF^{-1}a}{\phi}:=\bracket{a}{\cF^{-1} \phi}$ for all $\phi \in \cS$.  
That is, 
\[
Uf = \left(\cF^{-1}a\right) \ast f, 
\]
which holds for $f\in \cE'(X) \cup S(X)$ (see \cite[Section 3]{Am97}). 
 Next,  the methods of the singular integral operator theory allow showing the existence of an operator $T$ in $\cL(\E(X),\E(Y))$, which agrees with the operator $U$ on $S(X)$ (or $L^1_c(X)$).

For a Banach space $\E(X) \subset \cS'(X)$, which contains $S(X)$ as a dense subset, it immediately gives that $\E(X)$ is invariant with respect to $AU$. 
However, in the case when $\cS(X)$ is not dense in $\E(X) \subset \cS'_r(X)$, e.g., when $\E(X)$ is a Lorentz space $L^{p,\infty}(X)$ or a Besov type space $B^{s,q}_{L^{p,\infty}}(X)$ corresponding to it (see Section \ref{spaces}), it is not clear how to show directly that $Tf = U f$ for every $f\in \E(X)$ (cf. \cite[Problems 3.2 and 3.3]{HyWe06}, which arise in the study of the boundedness of Fourier multipliers on the classical Besov spaces $B^{s,q}_\infty(X)$, i.e., defined on the basis of $\Phi = L^\infty$). 

Such consistency of the solution operators with their bounded {\it extensions} resulting from the convolution representations is one of the main problems we address in the context of general Banach spaces considered in this paper; see Proposition \ref{bd multip} for the corresponding result. To solve it, we provide a suitable adaptation of the Rubio de Francia iteration algorithm, see Theorem \ref{RdeF}, which allows for relaxing some difficulties with density arguments that arise in such and related considerations in the literature; see, e.g. Remarks \ref{rem Lerner}(b) and \ref{rem RdF}(b).

The paper is organized as follows. In Section 2 we introduce the classes of Banach spaces with respect to which the questions on well-posedness and maximal regularity are examined. Their basic properties, which are crucial for the proofs of main results, are stated in Lemmas \ref{completness} and \ref{lem equiv norms}. 
 Section 3 contains the main extrapolation results, see Theorem \ref{ext general} and its specification to a class of Fourier multipliers, Proposition \ref{bd multip}.  
These results are applied in Section \ref{app} to derive the main results on the well-posedness and maximal regularity property of abstract integrodifferential equations \eqref{P}; see  Theorem \ref{main for itegro-dff} and Corollary \ref{for umd}. We conclude with some comments on particular cases of \eqref{P}, which may be of independent interest.

\section{Function spaces}\label{spaces}
Throughout, the symbol $\Phi$ is reserved to denote a Banach function space over $(\R, \ud t)$. We refer the reader to the monograph by Bennett and Sharpley \cite{BeSh88} for the background on Banach function spaces. Here, we mentioned only several facts we use in the sequel. 

Namely, each $\Phi$ is a Banach space, which is an order ideal of $L^0:=L^0(\R; \ud t)$, that is, for every $f\in L^0$ and $g\in \Phi$ if $|f|\leq |g|$, then $f\in \Phi$ and $\|f\|_\Phi \leq \|g\|$. 
Here, $L^0$ stands for the space of all complex measurable functions on $\R$ (as usual, any two functions equal almost everywhere are identified).
Moreover, $\Phi$ has Fatou's property, and by the Lorentz-Luxemburg theorem \cite[Theorem 2.7, p.10]{BeSh88}, $(\Phi')' = \Phi$ with equal norms. Here, $\Phi'$ denotes the {\it (K\"othe) dual} (or {\it associated space}) of $\Phi$; see \cite{BeSh88}.

We define the vector-valued variant of Banach function spaces $\Phi$ as follows. Let $X$ be a Banach space with norm $|\cdot|_X$. Set 
\[
\Phi(X):=\Phi (\R; \X ) :=\{ f:\R\rightarrow X \textrm{ strongly measurable}: \;\;  |f|_X\in \Phi\}
\]
and $\|f\|_{\Phi(X)} : = \||f|_X\|_{\Phi}$ for $f\in \Phi (X)$.

Moreover, we introduce a variant of vector-valued Besov and Triebel-Lizorkin spaces  corresponding to a Banach function space $\Phi$.  

 Let $\{\psi_j\}_{\N_0}$ be the resolution of the identity on $\R$ generated by a function $\psi\in \cD$ such that $\psi\equiv 1$ on $[-1,1]$ and $\supp \psi \subset [-2,2]$, i.e.  
\[
\psi_0 :=\psi, \qquad \psi_j := \psi(2^{-j} \cdot) - \psi(2^{-j+1} \cdot)\quad  \textrm{ for } j\in \N. 
\]
For $a\in \cO_M(\cL(X,Y)$, where $Y$ denotes another Banach space, we set
\[
a(D) f : = \cF^{-1}(\Theta(a, \cF f)) \quad \textrm{ for all $f$ in $\cS'(X)$,}
\] 
where $\Theta$ denotes a hypocontinuous, bilinear map from $O_M(\cL(X,Y))\times \cS'(X)$ into $\cS'(Y)$ (see the kernel theorem, e.g. \cite[Theorem 2.1]{Am97}). Here, $O_M(\cL(X,Y))$ stands for the space of slowly increasing smooth $\cL(X,Y)$-valued functions on $\R$. In particular, $a(D) \in \cL(\cS'(X), \cS'(Y))$ and for every $f\in \cS'(X)$ with $\cF f \in \cS'_r(X)$, $\Theta(a,\cF f) (t) = a(t) \cF f(t)$ $(t\in \R)$, and consequently $a(D) = \cF^{-1}(a(\cdot) \cF f)$.

Note that for every $f\in \cS'(X)$
\begin{equation}\label{resol on R}
\sum_{j\leq N} \phi_j(D)f \rightarrow f \qquad \textrm{in $\cS'(X)$ as $N\rightarrow \infty$}.
\end{equation}

Let $\Phi$ be a Banach function space. For all $s\in \R$ and $q\in [1,\infty]$ we set (with usual modification when $q=\infty$):
\[
B^{s,q}_{\Phi}(X):= \left\{f\in \cS'(X):\quad  \|f\|_{B^{s,q}_{\Phi}(X)} :=\bigl( \sum_{j=0}^\infty  \| 2^{sj} \psi_j(D) f \|^q_{\Phi(X)} \bigr)^{1/q}<\infty\right\}, \textrm{ and }
\]
%for  all $s\in \R$ and $q\in (1,\infty)$:
\[ 
F^{s,q}_{\Phi}(X) := \left\{ f\in \cS'(X): \quad  \|f\|_{F^{s,q}_{\Phi}(X)} :=\bigl\|  \bigl( \sum_{j=0}^\infty  | 2^{sj} \psi_j(D) f(\cdot)|^q_X \bigr)^{1/q}  \bigr\|_\Phi < \infty \right\}.
\]

In the case when $\Phi = L^p$ with $p\in [1,\infty]$ we get the classical vector-valued Besov and Triebel-Lizorkin spaces. For a coherent treatment in this vector setting see, e.g. \cite[Chapter VII]{Am95}, where historical remarks are included (see also Triebel \cite{Tr83} for the scalar case). Recall that, if $X$ is a Hilbert space, then $F^{m,2}_{L^p}(X) = W^{m, p}(X)$ for $m\in \N$ and $F^{s,2}_{L^p}(X) = H^{s, p}(X)$, $s\in \R$.  The case when $\Phi = L^p_w$ with $w$ in the Muckenhoupt class $A_\infty$ was recently considered, e.g. in \cite{MeVe12, MeVe15}; see also the references given therein. Here and in the sequel, by $L^p_w$ we denote the weighted Lebesgue space $L^p(\R, w\ud t; \C)$.

For general Banach function spaces $\Phi$ we need only basic properties of the corresponding spaces $B^{s,q}_\Phi(X)$ and $F^{s,q}_\Phi(X)$; see Lemmas \ref{completness} and \ref{lem equiv norms}. To keep the presentation of the main result of this paper transparent, we will not study these spaces on their own rights here.

\section{The extrapolation results}\label{sec extension}

The main results of this section, Theorems \ref{RdeF} and \ref{CZ bd}, set up a convenient framework for the study of boundedness properties of {\it distributional} Fourier multipliers discussed in the next section.

\subsection{The iteration algorithm}
We start with an adaptation of the Rubio de Francia iteration algorithm, which provides a crucial tool to resolve the consistency issue for solution operators, which was mentioned in Section \ref{intro}; see also Subsection~\ref{int represent}.

\begin{theorem}\label{RdeF} Let $\Phi$ be a Banach function space over $(\R, \ud t)$ such that the Hardy-Littlewood operator $M$ is bounded on $\Phi$ and its dual $\Phi'$. 
\begin{itemize}
\item [(i)] Let $X$ be a Banach space. Then, for every $p\in (1,\infty)$
\[
\Phi(X) \subset \bigcup_{w\in A_p}L^p_w(X)
\]
\item [(ii)] Let $X$ and $Y$ be Banach spaces and $p\in (1,\infty)$. Assume that $\{ T_j \}_{j\in J}$ is a family of linear operators  $T_j: \cS(X) \rightarrow\cS'(Y)$ such that for every  $\cW \subset A_p(\R)$ with $\sup_{w\in \cW} [w]_{A_p} < \infty$
\[
\sup_{w\in \cW} \sup_j \| T_j \|_{\cL(L^p_w(X),L^p_w(Y))} < \infty. 
\]
Then, each $T_j$ extends to a linear operator $\cT_j$ on  $ \bigcup_{w\in A_p} L^p_w(X) \subset L^1_{loc}(X)$ and has the restriction to an operator in $\cL(\Phi(X), \Phi(Y))$. Moreover, 
\[
\sup_{j\in J} \| \cT_j\|_{\cL(\Phi(X), \Phi(Y))} <\infty.
\] 
\end{itemize}
\end{theorem}

\begin{proof} Let $\cR$ and $\cR'$ denote the following sublinear, positive operators defined on $\Phi$ and $\Phi'$, respectively: 

\begin{equation}\label{algo ops}
\cR g := \sum_{k=0}^\infty \frac{M^k |g|}{(2\|M\|_{\Phi})^k} \quad (g\in \Phi) \qquad 
\cR' h := \sum_{k=0}^\infty \frac{M^k |h|}{(2\|M\|_{\Phi'})^k} \quad (h\in \Phi')  
\end{equation}

Here, $M^k$ stands for the $k$-th iteration of the Hardy-Littlewood operator $M$, $M^0 := I$, and $\|M\|_\Phi := \sup_{\|g\|_\Phi \leq 1}\|Mg\|_{\Phi}$. 
Since 
\[
M(\cR g ) \leq 2 \|M\|_\Phi \cR g  \quad (g\in \Phi) \quad \textrm{and} \qquad M(\cR' h) \leq 2 \|M\|_{\Phi'} \cR h \quad (h\in \Phi'),
\]
for every $h\in \Phi$ and $g\in \Phi'$ the functions $\cR g$ and $\cR'h$ belong to Muckenhoupt's class $A_1:=A_1(\R)$. Consequently, for every $p\in (1,\infty)$, $g\in \Phi$, $g\neq 0$, and $h\in \Phi'$, $h\neq 0$,  the function 
\begin{equation}\label{weight}
w_{g,h,p}:=(\cR g)^{1-p} \cR' h
\end{equation}
is in Muckenhoupt's class $A_p$ with the $A_p$-constant 
\begin{equation}\label{constants}
[w_{g,h,p}]_{A_p}\leq 2^p \|M\|_{\Phi}^{p-1} \|M\|_{\Phi'}.  
\end{equation}
For $(i)$, note that if $f\in \Phi(X)$, then $f\in L^p_w(X)$ for $w:= (\cR|f|_X)^{1-p} \cR'h$ with an arbitrary $h\in \Phi'$, $h\neq 0$. Indeed, since $|f|_X\leq \cR|f|_X$, we get 
\[
\int_\R |f|_X^p w \ud t \leq \int_\R \cR|f|_X \cR' h \ud t \leq \|\cR |f|_X\|_\Phi \| \cR'h\|_{\Phi'}<\infty.
\] 

For $(ii)$, first note that $\bigcup_{w\in A_p} L^p_w(X)$ is a subspace of $L^1_{loc}(X)$. Indeed, a direct computation shows that $\min(w,v) \in A_p$ if $w, v\in A_p$. Moreover, a standard approximation argument shows that for every $w\in A_p$ the Schwartz class $\cS(X)$ is a dense subset $L^p_w(X)$. Thus, each operator $T_j$ ($j\in J$) has an extension to an operator $T_{j,w}$ in $\cL(L^p_w(X), L^p_w(Y))$. Let $\cT_j:\bigcup_{w\in A_p} L^p_w(X)\rightarrow \bigcup_{w\in A_p} L^p_w(X)$ be given by
\[
\cT_j f := T_{j,w} f 
\]
for $w \in A_p$ such that $f\in L^p_w(X)$. These operators are well-defined and linear. 
Indeed, if $w, v\in A_p$ and $f\in L^p_w(X)\cap L^p_w(X)$, since $T_{j,w}$ is consistent with $T_{j,v}$  on $\cS(X)$ and there exists $(f_N)_{N\in \N}\subset \cS(X)$ such that $f_N\rightarrow f$ as $N\rightarrow \infty$ both in $L^p_w(X)$ and $L^p_v(X)$, $T_{j,w}f=T_{j,v}f$.  
By $(i)$, the operators $\cT_j$, $j\in J$,  are well-defined on $\Phi(X)$.

We show that $\cT_j$, $j\in J$, restrict to uniformly bounded operators in $\cL(\Phi(X), \Phi(Y))$.

Fix $f \in \Phi(X)$. Let $0\neq h\in \Phi'$ and 
$w:=w_{|f|_X, h, p}$ be given by \eqref{weight}. Then, since $|f|_X\leq \cR|f|_X$, 
\begin{align*}
\| \cT_j &\|_{\cL(L^p_w(X),L^p_w(Y))}  \|\cR\|_{\Phi} \|f\|_{\Phi(X)} \|\cR'\|_{\Phi'} \| h\|_{\Phi'}\\
& \geq  \| \cT_j\|_{\cL(L^p_w(X),L^p_w(Y))}  \|\cR |f|_X\|_\Phi \|\cR'h\|_{\Phi'} \\
& \geq \| \cT_j\|_{\cL(L^p_w(X),L^p_w(Y))}  \left( \int_\R \cR(|f|_X) \cR' h \, \ud t \right)^\frac{1}{p} \left(  \int_\R \cR(|f|_X) \cR' h \, \ud t \right)^\frac{1}{p'}\\
& \geq \| \cT_j\|_{\cL(L^p_w(X),L^p_w(Y))}  \left( \int_\R |f|_X^p \cR(|f|_X)^{1-p} \cR' h \, \ud t \right)^\frac{1}{p} \left(  \int_\R \cR(|f|_X) \cR' h \, \ud t \right)^\frac{1}{p'}\\
&\geq  \left( \int_\R |\cT_jf|_Y^p \cR(|f|_X)^{1-p} \cR' h \,\ud t \right)^\frac{1}{p} \left(  \int_\R \cR(|f|_X) \cR' h \ud t \right)^\frac{1}{p'}\\
&\geq \int_{\R} |\cT_jf|_Y \cR' h \, \ud t\\
&\geq \int_\R |\cT_jf|_Y |h| \, \ud t. 
\end{align*}
Our assumption on the dependence of the norms of $T_j$ and \eqref{constants} yield
$$
\mu: = \sup_{j\in J} \sup_{f\in \Phi(X), h\in \Phi'} \| \cT_j \|_{\cL(L^p_w(X),L^p_w(Y))} < \infty.
$$
Consequently, for every $h\in \Phi'$ and $f\in \Phi(X)$ we have that 
\[
\int_\R |\cT_jf|_Y |h| \ud t \leq \mu \|\cR\|_{\Phi} \|\cR'\|_{\Phi'}  \|f\|_{\Phi(X)} \|h\|_{\Phi'}. 
\]
Therefore, since $\Phi$ has Fatou's property, by the Lorentz-Luxemburg theorem $|\cT_jf|_Y\in \Phi''=\Phi$  and  $\|\cT_j\|_{\cL(\Phi(X), \Phi(Y))}\leq \mu \|\cR\|_{\Phi} \|\cR'\|_{\Phi'}$. This completes the proof. 

\end{proof}
 
\begin{remark}\label{rem RdF}
(a) The formulation of Theorem \ref{RdeF} corresponds to the idea of the proof of the Rubio de Francia extrapolation result, \cite[Theorem A]{RdeF87}, rather than to a modern framework provided in \cite{CrMaPe11}. However, in contrast to the proof given in \cite{RdeF87}, the above proof is constructive and adapts ideas presented in \cite[Chapter 4]{CrMaPe11}.

In a slightly less general form (see Corollary \ref{RdeF cor}), such modification of the Rubio de Francia extrapolation method was already applied to the study of the abstract Cauchy problems in \cite{ChKr17} and \cite{ChKr18}) (see also  \cite{Kr14}). The scalar variant of (i) (i.e. for $X=\C$) is also proved in \cite[Theorem 1.15]{KnMcCKa16} by a different argument than the one presented here.

(b)  Note that the point $(i)$ shows that if an operator $T$ is defined on $L^p_w(X)$ for every $w\in A_p$, then $Tf$ makes sense for every $f \in \Phi(X)$, where $\Phi$ satisfies the assumption of Theorem \ref{RdeF}. Therefore, no further extension procedure for such operator $T$ is needed. Cf. \cite[Section 5.4]{HaHa19}; see also Remark \ref{rem Lerner}(b) below.  
\end{remark}

The following result shows that Theorem \ref{RdeF} extends \cite[Theorem 4.10]{CrMaPe11}. 
We start with some preliminaries. 
A special class of Banach function spaces is the class of {\it symmetric spaces} (or {\it rearrangement invariant Banach function spaces}); see \cite{BeSh88}.
For a locally integrable weight $w$ on $\R$, we define $\|\cdot\|_{\E_w}:L^0_+ \rightarrow [0,\infty]$ by $\|f\|_{\E_w}: = \| f^*_w\|_{\E}$, where $f^*_w$ denotes the decreasing rearrangement of $f$ with respect to $w \ud t$. By \cite[Theorem 4.9, p. 61]{BeSh88}, the space 
\[
\E_w:= \{ f\in L^0: \|f\|_{\E_w} <\infty \}
\]
with $\| \cdot\|_{\E_w}$ is a symmetric space over $(\R, w \ud t)$. Moreover, $\E_w$ is a Banach function space over $(\R, \ud t)$ in the sense of Section \ref{spaces}.

For the convenience, we define the Boyd indices $p_\E$ and $q_\E$ of a symmetric space $\E$ following \cite{LiTz79} (in \cite{BeSh88} the Boyd indices are defined as the reciprocals with respect to \cite{LiTz79}). 

\begin{proposition}\label{RdeF cor} Let $\E$ be a symmetric space over $(\R,\ud t)$ with nontrivial Boyd indices $p_\E, q_\E \in (1,\infty)$. 
Then, for every Muckenhoupt weight $w\in A_{p_\E}$, the Hardy-Littlewood operator $M$ is bounded on $\E_w$ and its dual $(\E_w)'$ with respect to $(\R, \ud t)$. 
\end{proposition}

\begin{proof} By \cite[Lemma 4.12]{CrMaPe11}, $M$ is bounded on $\Phi = \E_w$.  By the inverse H\"older inequality, there exist $p$ and $q$ with $1<p<p_{\E}\leq q_{\E} <q <\infty$ such that $w\in A_p\subset A_q$. Therefore, by Muckenhoupt's theorem, $M$ is also bounded on $L^p_w$ and $L^q_w$.  
One can easily check that the operator $S(f) := M(f w)/w$ $(f\in L^1_{loc})$ is bounded on  
$ L^{p'}_{w}$  and $ L^{q'}_{w}$. Hence, it is of joint weak type $(q',q';p',p')$ with respect to $(\R, w\ud t)$; see \cite[Theorem 4.11, p. 223]{Ca66}. Consequently, 
since $p_{\E'} = q'_{\E}$ and $q_{\E'}= p'_\E$, we have that $1<q'<p_{\E'}\leq q_{\E'}< p'<\infty$.  Next, by Boyd's theorem \cite[Theorem 5.16, p.153]{BeSh88}, we get that $S$ is bounded on $(\E')_w$. Moreover, since $(\R, w\ud t)$ is a resonant space (see \cite[Theorem 1.6, p.51]{BeSh88}, we have  $(\E')_w = {(\E_w)'}^w$ with equal norms, where ${(\E_w)'}^w$ denotes the dual space of $\E_w$ with respect to $(\R, w\ud t)$. Indeed, the inclusion $(\E')_w \subset {(\E_w)'}^{w}$ follows from \cite[Proposition 4.2, p. 59]{BeSh88} and the reverse one from the Luxemburg representation theorem \cite[Theorem 4.10, p. 62]{BeSh88}.

Finally, note that for every $f\in (\E_w)'$  
\begin{align*}
\|Mf\|_{(\E_w)'} & = \|(Mf)w^{-1}\|_{{(\E_w)'}^w} = \| S(fw^{-1})\|_{(\E')_w}\\
& \leq \|S\|\|fw^{-1}\|_{{(\E')_w}} \\
& \leq \|S\| \|f\|_{{(\E_w)'}}, 
\end{align*}
where $\|S\|$ denotes the norm of $S$ on $(\E')_w$. This completes the proof. 
\end{proof}

\subsection{The singular integral operators}
To proceed, we specify a class of operators for which Theorem \ref{RdeF} applies. In particular, it allows us to derive the basic properties of generalized Besov and Triebel-Lizorkin spaces introduced in Section \ref{spaces}, which are involved in the proofs of the subsequent results. 

We say that a bounded linear operator $T$ in $\cL(L^p(X), L^p(Y))$ for some $p\in (1,\infty)$ is a {\it singular integral operator} if there exists a function $k\in L^1_{loc}(\dot \R; \cL(X,Y))\cap \cS'(\cL(X,Y))$ such that for all $f\in L^\infty_c(X)$ and all $t\notin \supp f$
\[
Tf(t) = \int_\R k(s) f(t-s) \ud s.
\]
Here, $L^\infty_c(X)$ stands for the space of all $X$-valued, essentially bounded, measurable functions with compact support in $\R$. 
If, additionally, $k \in \cC^1(\dot \R; \cL(X,Y))$ and 
\begin{equation}\label{con K}
[k]_{\cK_{1}} := \max_{l=0,1} \sup_{t\neq 0} \|t^{l+1} k^{(l)}(t)\|_{\cL(X,Y)}<\infty,
\end{equation}
then we say that $T$ is a Calder\'{o}n-Zygmund operator.  
The boundedness of Calder\'{o}n-Zygmund operators on different types of vector-valued Banach function spaces attracted attention in the literature. Here, we mention two pioneering papers \cite{BeCaPa62}, and \cite{RuRuTo83}, from which some ideas we reproduce below.

For instance, by direct analysis of the proof of \cite[Theorem 1.6, Part I]{RuRuTo83} each Calder\'on-Zygmund operator satisfies the assumptions of Theorem \ref{RdeF}(ii) (it is also true for a larger class of singular integral operators; see \cite[Theorem 7]{ChKr17}). The following result makes this statement precise.

\begin{corollary}\label{CZ bd} Let $X$ and $Y$ be Banach spaces. 
Let $\{T_j\}_{j\in J}$ be a family of Calder\'on-Zygmund operators  with kernels $k_j$, $j\in J$, such that 
\begin{equation}\label{ub for CZ}
\sup_{j\in J}\|T_j\|_{\cL(L^p(X), L^p(Y))} < \infty \quad \textrm{and} \quad \sup_{j\in J}[k_j]_{\cK_1}<\infty.
\end{equation} 
Then, $\{T_j\}_{j\in J}$ satisfies the assumptions of Theorem \ref{RdeF}$(ii)$, that is, 
for every $q\in (1,\infty)$ and 
for every $\cW\subset A_{q}$ with $\sup_{w \in \cW} [w]_{A_q}<\infty$ 
\[
 \sup_{j\in J, w\in \cW }\| T_j\|_{\cL(L^q_w(X), L^q_w(Y)}<\infty.
\]
Consequently, the conclusion of Theorem \ref{RdeF}$(ii)$ holds for $\{T_j\}_{j\in J}$. 

In particular, if $\{\psi_j\}_{j\in \N_0}$ is a resolution of the identity on $\R$ and $\Phi$ satisfies the assumption of Theorem \ref{RdeF}, then the operators $\psi_j(D)$, $j\in \N_0$, restrict to uniformly bounded operators in $\cL(\Phi(X))$. 
\end{corollary}

\begin{proof}
By direct analysis of the constants involved in the main ingredients of the proof of \cite[Theorem 1.6, Part I]{RuRuTo83}, the norms $\|T\|_{\cL(L^q_w(X),L^q_w(Y))}$ of each Calder\'{o}n-Zygmund operator $T$, depend on $q\in (1,\infty)$, the constant from the $\cK_1$-condition of its kernel, and the norm of $T$ as an operator in $\cL(L^p(X), L^p(Y))$, and finally they are bounded from above if $w$ varies in a subset of $A_q$ on which $[w]_{A_q}$ are uniformly bounded. Cf. also \cite{Hy12}, or \cite{Le13}, for the precise dependence of $L^q_w$-norms of Calder\'on-Zygmund operators on the $A_q$-constants $[w]_{A_q}$ of Muckenhoupt's weights $w\in A_q$.

For the last statement, since $\cF^{-1} \psi_j(t)=2^{-j}\cF^{-1}\psi (2^{-j}t) - 2^{-j+1}\cF^{-1}\psi (2^{-j+1}t)$ $(t\in \R, j\in \N_0)$, it is straightforward to see that  $\sup_{j\in \N_0}[\cF^{-1}\psi_j]_{\cK_1}<\infty$ and $\sup_{j\in \N_0}\|\cF^{-1}\psi_j\|_{L^1}< \infty$. 
Therefore, by Young's inequality, we get \eqref{ub for CZ} for $\psi_j(D)$, $j\in \N_0$. 
Alternatively and more directly, this statement can be also derived by the argument presented in the proof of Proposition \ref{bd multip}(i) below.
\end{proof}

\begin{remark}\label{rem Lerner}
(a)
In the scalar case, i.e. $X=Y=\C$, the boundedness of Calder\'on-Zygmund operators on Banach function spaces $\Phi$ such that $M$ is bounded on $\Phi$ and $\Phi'$ was already stated in \cite[Proposition 6]{Ru14}. The proof given in \cite{Ru14} is based on a formal application of Coifman's inequality. Theorem \ref{RdeF} allows us to provide an alternative, direct proof of this fact. %, which is direct and rigorous.  

(b) {\it The approach via Str\"omberg's sharp maximal operator.}
A conceptually different approach to the study of the boundedness of Calder\'on-Zygmund operators was presented by Jawerth and Torchinsky \cite{JaTo85}. It relies on the local sharp maximal function operator ($\alpha \in (0,1))$:
\[
M^\sharp_\alpha f(t):= \sup_{Q\ni t} \inf_{z\in \C} \inf\bigl\{ \beta \geq 0: |\{ s\in Q: |f(s) - z| > \beta\}| <\alpha |Q| \bigr\}
\] for every $f\in L^1_{loc}$ and $t\in \R$.

It is shown in \cite[Theorem 4.6]{JaTo85} (in the scalar case, but the vector one follows easily) that, in particular,  if $T$ is a Calder\'on-Zygmund operator,  then there exist  constants $\alpha$ and $\mu$ such that
\begin{equation}\label{JT est}
M^\sharp_\alpha (|Tf|_Y) \leq \mu M(|f|_X) 
\end{equation}
 for all {\it appropriate } $f\in L^1_{loc}(X)$. Here, the word {\it 'appropriate'} plays a crucial role to apply this inequality for a further study of the boundedness of $T$ on some function spaces. By an analysis of the proof, we see that a priori one can consider the set of all $f\in L^1_{loc}(X)$ such that  $Tf \in L^0(X)$ with $(|Tf|_Y)^*(+\infty) = 0$ and for every interval $I\subset \R$, $Tf^I\in L^1(I)$ with $f^I:=f\chi_{\R\setminus 2I}$, and 
\[
(Tf^I)(t) = \int_{\R \setminus 2I} k(t-s)f(s) \ud s \,\, \textrm{  for a.e. } t \in I. 
\]
Here, for a function $g\in L^0$, $g^*(+\infty) = 0$ if and only if $|\{ t\in \R: |g(t)|>\alpha \}|<\infty$ for all $\alpha>0$. 
Denote the above set by $D_{T}$ and set $D_{T,\Phi} :=D_{T}\cap \Phi(X)$ for a Banach function space $\Phi \subset L^1_{loc}$. Of course we have that $L^1_c(X) \subset D(T)$.

Moreover, recall  Lerner's characterization of boundedness of $M$ on a $\Phi'$. Namely, in \cite[Corollary 4.2 and Lemma 3.2]{Le10} Lerner proved that if the maximal operator $M$ is bounded on a Banach function space $\Phi$, then $M$ is bounded on $\Phi '$ if and only if
there exists a constant $\mu$ such that for every $f\in L^0$ with $f^*(+\infty) = 0$  
\[
\|Mf\|_\Phi \leq \mu\|M^\sharp f\|_\Phi.
\]

Note that by the Jawerth-Torchinsky pointwise estimate \eqref{JT est} and Lerner's characterisation, we get for any Banach function space $\Phi$ such that $M$ is bounded on $\Phi$ and $\Phi'$ that for every $f\in D_{T, \Phi}$
\begin{equation}\label{bd est}
\|Tf\|_{\Phi(Y)} \leq \|M(|Tf|_Y)\|_{\Phi} \lesssim \|M^\sharp_{\alpha} (|Tf|_Y)\|_{\Phi} \lesssim \| M(|f|_X)\|_\Phi\lesssim \|f\|_{\Phi(X)}.
\end{equation}
But it is not clear how to show directly (without the use of the Rubio de Francia algorithm) that $D_{T,\Phi}$ is dense in $\Phi$.

Again, Theorem \ref{RdeF}(i) solves this density issue (and, in particular, relaxes a density assumption in \cite[Proposition 9]{Ru14}). 

In fact, let $f\in \Phi(X)$ and set $f_N = f\chi_{[-N,N]}$ for all $N\in \N$.
By Theorem \ref{RdeF},  $f\in L^2_w(X)$ and $T$ in  $\cL(L^2_w(X), L^2_w(Y))$ for some weight $w\in A_2$. Then, since $f_N\rightarrow f$ in $L^2_w(X)$, we have for some sequence $(N_l)_l\subset \N$ that $Tf_{N_l}(t) \rightarrow Tf(t)$ for a.e. $t\in \R$. Moreover, $f_{N_l} \in \Phi\cap L^1_c(X)$, so by \eqref{bd est},
\[
\|Tf_{N_l}\|_{\Phi(Y)} \lesssim \|f_{N_l}\|_{\Phi(X)}\lesssim \|f\|_{\Phi(X)}.
\]
Therefore, the Fatou property of $\Phi$ implies that $Tf\in \Phi(Y)$ and $\|Tf\|_{\Phi(Y)} \lesssim \|f\|_{\Phi(X)}$. 

\end{remark}

\subsection{The basic properties of generalised Besov and Triebel-Lizorkin spaces}

We start with an auxiliary lemma on dense subsets of $B^{s,q}_\Phi(X)$  and $F^{s,q}_\Phi(X)$.

\begin{lemma}\label{completness}
Let $X$ be a Banach space and let $\Phi$ be a Banach function space over $(\R, \ud t)$ such that the Hardy-Littlewood operator $M$ is bounded on $\Phi$ and its dual $\Phi'$.

Then, for every 
\[
\E \in \left\{ B^{s,q}_\Phi,\, F^{s,r}_\Phi: s\in \R, q\in [1,\infty], r \in (1,\infty) \right\}
\]
\begin{equation}\label{embedding}
\cS(X)\subset \E(X) \hookrightarrow \cS'(X),
\end{equation}
that is, $\E(X)$ embeds continuously into $\cS'(X)$. In particular, $\E(X)$ is a Banach space.

Furthermore, the set $\Phi(X) \cap \E(X)$ is dense in $\E(X)$ for each $\E$ with $q<\infty$, and in the other cases when $\E=B^{s,\infty}_\Phi$ $(s\in \R)$,  $\Phi(X) \cap B^{s,\infty}_\Phi(X)$ is $\sigma(B^{s,\infty}_\Phi(X), B^{-s,1}_{\Phi'}(X^*))$-dense in $B^{s,\infty}_\Phi(X)$. More precisely, for every $f\in \E(X)$  we have that 
\[ \sum_{0\leq j\leq N} \psi_j(D) f = \psi(2^{-N} D)f \rightarrow f  \quad \textrm{ as } N\rightarrow \infty
\]  
in $\E(X)$ for each $\E$ with $q<\infty$, and for $\E = B^{s,\infty}_\Phi$, $s\in \R$, this convergence holds in $\sigma(B^{s,\infty}_\Phi(X), B^{-s,1}_{\Phi'}(X^*))$-topology.
\end{lemma}
Note that the following duality pairing for $f \in B^{s,\infty}_{\Phi}(X)$ and $g\in B^{-s,1}_{\Phi'}(X^*)$ $(s\in \R)$ 
\begin{equation}\label{pairing}
\bracket{g}{f}:= \sum_{j,l\in \N_0} \int_{\R} \bracket{\psi_l(D)g(t)}{\psi_j(D)f(t)}_{X^*,X} \ud t,
\end{equation}
yields a natural embedding of $B^{-s,1}_{\Phi'}(X^*)$ into $(B^{s,\infty}_{\Phi}(X))^*$. 

As usual, the  $\sigma(B^{s,\infty}_\Phi(X), B^{-s,1}_{\Phi'}(X^*))$-topology is the weakest topology on $B^{s,\infty}_\Phi(X)$ for which every $g\in  B^{-s,1}_{\Phi'}(X^*)$ becomes continuous; see also Remark \ref{norms} below.

\begin{proof}
Let $\{\psi_j\}_{j\in \N_0}$ be a resolution of the identity on $\R$. By Corollary \ref{CZ bd}, $\psi_j(D)$, $j\in \N_0$, restrict to uniformly bounded operators in $\cL(\Phi(X))$ and in $\cL(\Phi')$. 
Set $\chi_j:= \psi_{j-1} + \psi_j + \psi_{j+1}$, $j\in \N_0$, with $\psi_{-1}\equiv 0$.
Since for every measurable set $A\subset \R$ with $|A|>0$ there exists a constant $\delta>0$ such that $M \chi_A\geq \frac{\delta}{\max(1, t)}$ $(t\in \R)$, by the ideal property of $\Phi$ and $\Phi'$ and our assumption on the boundedness of $M$ we get that $\cS\subset \Phi \cap \Phi'$.

{\it The case $\E = B^{s,q}_\Phi$.}   Fix $s\in \R$ and $q\in [1,\infty)$.
Let $(f_l)_{l\in \N}$ be a sequence in $B^{s,q}_{\Phi}(X)$ with $\| f_l\|_{B^{s,q}_{\Phi}(X)}\rightarrow 0$ as $l\rightarrow \infty$. Since $\{\psi_j(D)\}_{j\in \N_0}$ is a resolution of the identity on $\cS'(X)$ (see \eqref{resol on R}) and $\Phi(X) \subset \cS'_r(X)$, by H\"older's inequality, for every $\phi \in \cS$ we get 
\begin{align*}
| f_l(\phi) |_X & = \left|\sum_{j\in \N_0} \psi_j(D) f_l (\phi)\right|_X  = \left| \sum_{j\in \N_0} \psi_j(D) f_l [\cF(\chi_j \cF^{-1}\phi)]\right|_X \\
& = \left|  \sum_{j\in \N_0} \int_\R  \psi_j(D) f_l (t) (\chi_j(D)\breve\phi)(t) \ud t \right|_X\\
&  \leq  \sum_{j\in \N_0} \int_\R \left|  \psi_j(D) f_l (t) \right|_X |(\chi_j(D)\breve\phi)(t)| \ud t\\ 
& \leq \sum_{j\in \N_0} \left\| 2^{js} \psi_j(D) f_l \right\|_{\Phi(X)} \| 2^{-js} \chi_j(D)\breve\phi\|_{\Phi'} \\ 
& \leq \left( \sum_{j\in \N_0} \left\| 2^{jsq} \psi_j(D) f_l \right\|^q_{\Phi(X)} \right)^{1/q} 
\left( \sum_{j\in \N_0} \| 2^{-js} \chi_j(D)\breve\phi\|_{\Phi'}^{q'} \right)^{1/q'} \\ 
& \leq  \|f_l\|_{B^{s,q}_{\Phi}(X)} \left( \sum_{j\in \N_0} \| 2^{-js} \chi_j(D)\check\phi\|_{\Phi'}^{q'} \right)^{1/q'}  
\end{align*}
Here, we set $\breve{\phi}$ for $\phi(-\cdot)$.
Note that in the case when $s>0$ the proof is complete. For $s\leq 0$, let $\alpha > |s|+1$ and $\rho(t):= (1+t^2)^{-\alpha}$, $t\in \R$. 
It is easy to check that $2^{(|s|+1) j} \cF^{-1}(\rho\chi_j),{j\in \N_0},$ uniformly satisfy  the $\cK_1$-condition (see \eqref{con K}) and their $L^1$-norms are uniformly bounded too. By Proposition \ref{CZ bd}, $(2^{(|s|+1) j}\chi_j \rho)(D)$, $j\in \N_0$, restrict to uniformly bounded operators in $\cL(\Phi')$.

 Therefore, since  $\cF^{-1}(\rho^{-1} \cF^{-1}\check \phi)\in \cS \subset \Phi'$,  for every $j\in \N_0$ we get that
\begin{align*}
\| 2^{j|s|} \chi_j(D)\phi \|_{\Phi'} & \leq 2^{-j} \| (2^{(|s|+1)j} \rho \chi_j)(D)[\cF^{-1}(\rho^{-1} \cF^{-1}\check \phi)]  \|_{\Phi'}\\
& \leq 2^{-j} \| (2^{(|s|+1)j} \rho \chi_j)(D)\|_{\cL(\Phi')} \|[\cF^{-1}(\rho^{-1} \cF^{-1}\check \phi)]  \|_{\Phi'}.
\end{align*}
It completes the proof in this case. 

The proof of the last assertion makes the use of similar arguments to those applied above. For $f\in B^{s,q}_{\Phi}(X)$ and $N\in \N$ set $f_N:=\sum_{j\leq N}\psi_j(D) f\in \Phi(X)$. Since  $\{\psi_j(D)\}_{j \in \N_0}$ is the resolution of the identity operator on $\cS'(X)$ (see \eqref{resol on R}),  we have that $f=  f_N +  \sum_{j> N}\psi_j(D) f$. Moreover, since the operators $\{ \phi_j(D) \}_{\in \N_0}$ are uniformly bounded in $\cL(\Phi(X), \Phi(X))$, we get for $q<\infty$ that
\begin{align*}
\left\|f -  f_N \right\|_{B^{s,q}_{\Phi}(X)} & =  \left(  \sum_{j\in \N_0} 2^{sqj} \| \psi_j(D) \sum_{l>N} \psi_l(D)f \|^q_{\Phi(X)} \right)^{1/q}\\
& =  \Big( 2^{sqN} \| (\psi_N \psi_{N+1})(D)f \|^q_{\Phi(X)}\\
& + 2^{sq(N+1)}\| (\psi_{N+1}(\psi_{N+1}+\psi_{N+2}))(D)f \|^q_{\Phi(X)}\\
& + \sum_{j\geq N+2} 2^{sqj}\| (\psi_j \chi_j)(D) f\|^q_{\Phi(X)} \Big)^{1/q}\\
& \lesssim  \Big( \sum_{j\geq N} 2^{sqj} \| \psi_j(D) f\|^q_{\Phi(X)} \Big)^{1/q} \rightarrow 0 \textrm{ as } N\rightarrow \infty. 
\end{align*}
For $q=\infty$, note first that for every $j, l\in \N_0$ with $|j-l|>1$ 
\[
\bracket{\psi_l(D)g}{\psi_j(f-f_N)}_{\Phi'(X^*),\Phi(X)} = 0.
\]
Indeed, since $\breve{\psi}_{j}(D), \psi(D)\in \cS'(X^*)$ and $\cS(X^*)$ is dense in $\cS'(X^*)$, we get that 
\[
\bracket{\psi_l(D)g}{\psi_j(D)(f-f_N)}_{\Phi'(X^*),\Phi(X)}  = \bracket{\breve{\psi}_j(D)\psi_l(D)g}{\chi_j(D)(f-f_N)}_{\Phi'(X^*),\Phi(X)} = 0.
\]
Consequently, for $N>2$ we have that 
\begin{align*}
\left|\bracket{g}{f-f_N}\right| & = \sum_{r\in \{0, \pm 1\}} \sum_{l \geq N-2} \left|\bracket{\psi_l(D)g}{\psi_{l+r}(D)(f-f_N)}_{\Phi'(X^*),\Phi(X)}\right|\\
& \leq 3 \sup_{j\in \N_0}\| \chi_j(D)\|_{\cL(\Phi(X))} \|f\|_{B^{s,\infty}_{\Phi}(X)} \sum_{l\geq N-2} 2^{-sl} \|\psi_l(D) g\|_{\Phi'(X^*)}.
\end{align*}
Therefore, the proof of this case is complete.

{\it The case $\E= F^{s,q}_\Phi$.} Fix $s\in \R$ and $q\in (1, \infty)$.
Let $f\in F^{s,q}_\Phi(X)$ and $f_N:= \sum_{j\leq N} \psi_j(D)f$ for all $N\in \N$. Note that $f_N \in \Phi(X) \cap F^{s,q}_\Phi(X)$ and 
\[
\|f-f_N\|_{F^{s,q}_\Phi(X)} \lesssim  \sum_{r=-1}^1 \Big\|\big( \sum_{j\geq N} 2^{jsq}|\psi_{j+r}(D)\psi_j(D)f|_X^q\big)^{1/q}\Big\|_{\Phi}.
\]
Set 
\[
\widetilde Gf := \widetilde G_{N,r}f := \big( \sum_{j\geq N} 2^{jsq}|\psi_{j+r}(D)\psi_j(D)f|_X^q\big)^{1/q}, 
\]
\[
 Gf : =  G_{N}f := \big( \sum_{j\geq N} 2^{jsq}|\psi_j(D)f|_X^q\big)^{1/q}\quad (r\in \{\pm 1 ,0\}, N\in \N).
\]
Let $h\in \Phi'$ with $h\neq 0$ and $w:= w_{Gf, h}= \cR(Gf)^{1-q} \cR' h$; see \eqref{weight}.
Then, by similar arguments as in the proof of Theorem \ref{RdeF}, we get 
\begin{align*}
\sup_j\| \psi_j(D) &\|_{\cL(L^q_w(X))}  \|\cR\|_{\cL(\Phi)} \|Gf\|_\Phi \|\cR'\|_{\cL(\Phi')} \| h\|_{\Phi'}\\
& \geq \sup_j\| \psi_j(D) \|_{\cL(L^q_w(X))} \|\cR Gf\|_\Phi \|\cR'h\|_{\Phi'} \\
& \geq \sup_j\| \psi_j(D) \|_{\cL(L^q_w(X))} \left( \int_\R \cR(Gf) \cR' h \, \ud t \right)^\frac{1}{q} \left(  \int_\R \cR(Gf) \cR' h \, \ud t \right)^\frac{1}{q'}\\
&\geq  \left( \int_\R (\widetilde Gf)^q w \,\ud t \right)^\frac{1}{q} \left(  \int_\R \cR(Gf) \cR' h \ud t \right)^\frac{1}{q'}\\
&\geq \int_{\R} \widetilde Gf  \cR' h \, \ud t\\
&\geq \int_\R \widetilde Gf |h| \, \ud t.
\end{align*}
Since 
$\sup_{f\in F^{s,q}_{\Phi}(X), h\in \Phi'} [w_{Gf,h}]_{A_q} <\infty$, by the Lorentz-Luxemburg theorem, we get that $\widetilde G f\in \Phi(X)$ and for all $r\in \{\pm 1 ,0\}, N\in \N$ we have 
\[
\| \widetilde G_{N,r} f\|_\Phi \leq \sup_{0\neq h\in \Phi'}\sup_j\| \psi_j(D) \|_{\cL(L^q(\R, w_{Gf,h}\ud t; X)}  \|\cR\|_{\cL(\Phi)} \|G_Nf\|_\Phi \|\cR'\|_{\cL(\Phi')}. 
\]
Therefore, $f_N\rightarrow f$ in $F^{s,q}_\Phi(X)$ as $N\rightarrow \infty$. The proof of \eqref{embedding} follows the arguments already presented above. 
\end{proof}

\begin{remark}\label{norms}
a) Applying similar arguments to those from the above proof, one can easily show that the definition of Besov and Triebel-Lizorkin type spaces introduced in Section \ref{spaces} is independent of a particular choice of a function $\psi$. 
The same refers to the definition of the duality paring \eqref{pairing}.
We omit the proof.

b) Direct arguments based on the ideal property of $\Phi'$ shows that $\Phi'(X^*)$ is separating for $\Phi(X)$, i.e. for each $f\in \Phi(X)$, if for all $g\in \Phi'(X)$ 
\[
\int_{\R} \bracket{g(t)}{f(t)}_{X^*,X} \ud t = 0, 
\] 
then $f = 0$. Since for every $f\in B^{s,\infty}_{\Phi}(X)$ and $g\in \Phi'(X^*)$, $g_N:= \psi(2^{-N}D)g \in B^{-s,1}_{\Phi'}(X^*)\cap \Phi'(X^*)$, and for all $l \in \N_0$ and $N > l+1$
\[
\bracket{g-g_N}{\psi_l(D)f} = \bracket{\breve{\psi_l}(D)(g-g_N)}{\chi_l(D)f}=0, 
\]
we get that $B^{-s,1}_{\Phi'}(X^*)$ is also separating on $B^{s,\infty}_{\Phi}(X)$.

\end{remark}

Now we are in the position to extend Theorem  \ref{RdeF}(ii). This extension is applied in the study of the extrapolation of $L^p$-maximal regularity under additional geometric conditions on the underlying Banach space $X$; see Remark \ref{R bound}(b) and Corollary \ref{for umd}.  

\begin{theorem}\label{ext general} Let $X$ and $Y$ be Banach spaces and $p\in (1,\infty)$. Let $\Phi$ be a  Banach function space such that the Hardy-Littlewood operator $M$ is bounded on $\Phi$ and its dual $\Phi'$.

Let $T:\cS(X) \rightarrow\cS'(Y)$ be a linear operator such that  for all $f\in \cS(X)$
\begin{equation}\label{commute}
T\psi_j(D)f = \psi_j(D) Tf \quad\quad (j\in \N_0).
\end{equation}
Then the following assertions hold.
\begin{itemize}
\item [(i)] If for every  $\cW \subset A_p$ with $\sup_{w\in \cW} [w]_{A_p} < \infty$
\[
\sup \left\{ \| T \psi_j(D) f \|_{L^p_w(Y)}: f\in \cS(X) \textrm{ with } \|f\|_{L^p_w(X)}\leq 1,   w\in \cW \textrm{ and } j\in \N_0 \right\} < \infty,
\]
then for every $s\in \R$ and $q\in [1,\infty]$ the operator $T$ has {an}~extension to {an}~operator $\cT_\Phi$ in  $\cL(B^{s,q}_\Phi(X), B^{s,q}_\Phi(Y))$.

Moreover, if $\Psi$ is another Banach function space such that $M$ is bounded on $\Psi$ and $\Psi'$, then $\cT_\Psi f = \cT_\Phi f$ for all $f\in 
B^{s,q}_\Phi(X)\cap B^{s,q}_\Psi(X)$. 

The same conclusion holds for Triebel-Lizorkin spaces $F^{s,q}_\Phi$ with $s\in \R$ and $q\in (1,\infty)$.

\item [(ii)] 
If for every  $\cW \subset A_p$ with $\sup_{w\in \cW} [w]_{A_p} < \infty$
\[
\sup \left\{ \| T f \|_{L^p_w(Y)}: f\in \cS(X) \textrm{ with } \|f\|_{L^p_w(X)}\leq 1,   w\in \cW \right\} < \infty,
\]
then for every 
\[
\E \in \left\{\Phi,\,B^{s,q}_\Phi, \, F^{s,r}_\Phi: s\in \R, q\in [1,\infty], r\in (1,\infty) \right\} 
\]
the operator $T$ extends to an operator $\cT_\E$ in $\cL(\E(X), \E(Y))$.

Moreover, all such extensions are consistent, that is, if $\Phi_i$, $i=1,2,$ are Banach function spaces such that $M$ is bounded on them and their duals 
and 
\[
\E, \widetilde\E \in \left\{\Phi_i,\, B^{s,q}_{\Phi_i},\,  F^{s,r}_{\Phi_i}: s\in \R, q\in [1,\infty], r\in (1,\infty),  i=1,2\right\},
\]
 then $\cT_\E f = \cT_{\widetilde \E} f$ for all $f\in \E(X) \cap \widetilde \E (X)$. 

\end{itemize}

\end{theorem}

\begin{proof} $(i)$  
By Theorem \ref{RdeF}(ii) the operators $T\psi_l(D)$, $l\in \N_0$, extend to uniformly bounded operators $\cT_l:=\cT_{l,\Phi}$, $l\in \N_0$, in $\cL(\Phi(X), \Phi(Y))$.  
Let $q<\infty$.
We show that the operators $\sum_{l\leq N} \cT_l$, $N\in \N$, are uniformly bounded in $\cL(B^{s,q}_\Phi(X), B^{s,q}_\Phi(Y))$ and that for every $f\in \Phi(X)\cap B^{s,q}_\Phi(X) $
\[
\lim_{N\rightarrow \infty} \sum_{l\leq N} \cT_l f
\]
exists in $B^{s,q}_\Phi(Y)$. Since, by Lemma \ref{completness}, $\Phi(X)\cap B^{s,q}_\Phi(X)$ is a dense subset of $B^{s,q}_\Phi(X)$, the series $\sum_{l\in \N_0} \cT_l$ defines an operator $\cT_\Phi$ in $\cL(B^{s,q}_\Phi (X), B^{s,q}_\Phi (Y))$.  

First, note that by \eqref{commute},  
for all  $f\in S(X)$ and  all $j, l\in \N_0$, we have
\[
\psi_j(D)T\psi_l(D) f = T \psi_l(D)\psi_j (D) f = T (\psi_l \psi_j ) (D) f .
\]
Since $S(X)$ is a dense subset of $L^p_w(X)$ for all $w\in A_p$, by Lemma \ref{RdeF}(i), for every $f\in \Phi(X)$ we get
\[
\psi_j(D)\cT_l f = \cT_l \psi_j(D) f
\]
with $\psi_j(D)\cT_l f = 0 $, when $|j-l|>1$.

Further, for all $f\in \Phi(X)\cap B^{s,q}_\Phi(X)$ and all integers $0\leq N < N'$ we get (with usual modification for $q=\infty$)
\begin{align*}
\left\|\sum_{N \leq l\leq N' } \cT_l f \right\|_{\B^{s,q}_{\Phi}(Y)} 
& \leq \left(  \sum_{j}2^{jsq} \left \| \psi_j(D) \sum_{N\leq l\leq N'} \cT_l f \right \|_{\Phi(Y)}^q \right)^{1/q}\\
& \leq \left(  \sum_{N-1\leq j \leq N'+1}2^{jsq} \left \| \sum_{l=j-1}^{j+1} \cT_l \psi_j(D) f \right \|_{\Phi(Y)}^q \right)^{1/q}\\
&\lesssim \left(  \sum_{N-1\leq j \leq N'+1}2^{jsq} \left \| \psi_j(D) f \right \|_{\Phi(X)}^q \right)^{1/q}, 
\end{align*} where we set $\psi_{-1} \equiv 0$, when $N=0$. %The constant $m$ above can by choosen be independent on $N, N'$ and $f$. 
Therefore, our claim holds.

Since $\{\psi_j(D)\}_{j \in \N_0}$ is a resolution of the identity operator on $\cS'(Y)$, that is,  for every $g\in \cS'(Y)$
\[
\lim_{N \rightarrow \infty} \sum_{0 \leq j \leq N }\psi_j(D) g = g  \textrm{ in } \cS'(Y),\]
we get that $\cT_\Phi f =  Tf$ for all $f\in \cS(X)$. 

Further, if $q<\infty$, then for $f\in B^{s,q}_\Phi(X)\cap B^{s,q}_\Psi(X)$, $\psi(2^{-N} D) f \in \Phi(X)\cap \Psi(X)$ $(N\in \N)$ and  $\psi(2^{-N} D) f \rightarrow f$ both in $B^{s,q}_\Phi(X)$ and $B^{s,q}_\Psi(X)$; see Lemma \ref{completness}. Since, by Theorem \ref{RdeF}(ii), $\cT_{l,\Phi}$ is consistent with $\cT_{l,\Psi}$ on $\Phi(X)\cap \Psi(X)$, we get that $\cT_\Phi \psi(2^{-N} D) f = \cT_\Psi \psi(2^{-N} D) f$. 
If $q=\infty$, then note that $\cT_l f= \cT_l \chi_l(D) f$ for every $f\in \Phi(X)$ and $l\in \N_0$. Moreover, $\cT_l \chi_l(D)\in \cL(B^{s,\infty}_{\Phi}(X), B^{s,\infty}_{\Phi}(Y))$ uniformly in $l\in \N_0$. It is readily seen that for every $f\in B^{s,\infty}_{\Phi}(X)$ the series $\sum_{l\in \N_0} \cT_l\chi_l(D) f$ converges in $\cS'(Y)$ to an element $\cT_\Phi f$ of $B^{s,\infty}_\Phi(Y)$ such that 
\[
\|\cT_\Phi f\|_{B^{s,\infty}_\Phi(Y)}\leq C \sup_{j} \| \cT_l \|_{\cL(\Phi(X), \Phi(Y)} \|f\|_{B^{s,\infty}_\Phi(X)},
\] where $C$ is a constant independent of $f$. Therefore, $\cT_\Phi$ is in $\cL(B^{s,\infty}_{\Phi}(X), B^{s,\infty}_{\Phi}(Y))$. For the consistency of $\cT_\Phi$ and $\cT_\Psi$, note that if $f\in B^{s,\infty}_{\Phi}(X)\cap B^{s,\infty}_{\Psi}(X)$, then 
\begin{align*}
\cT_\Psi f & = \cS'(Y)-\sum_{l} \cT_{l,\Psi} \chi_l(D) f  = \cS'(Y)-\lim_{N\rightarrow \infty} \sum_{l\leq N} \cT_{l,\Psi} \chi_l(D) \psi(2^{-N-2}D)f\\
& = \cS'(Y)-\lim_{N\rightarrow \infty} \sum_{l\leq N} \cT_{l,\Phi} \chi_l(D) \psi(2^{-N-2}D)f = \cT_{\Phi} f.
\end{align*} This completes the proof of $(i)$ for Besov spaces.
\\
For the last statement of $(i)$, first note that $F^{s,q}_\Psi = B^{s,q}_\Psi$ for all $\Psi=L^q_w$ with $q\in (1,\infty)$ and $w\in A_q$. Therefore, for such $\Psi$, $T$ extends to an operator $\cT_\Psi$ in $\cL(F^{s,q}_\Psi(X),F^{s,q}_\Psi(Y)$ and its norm is bounded by 
\begin{align*}
\mu_{w}  :=\mu_{w,q} &:=\sup_{l\in \N_0} \|\cT_{l,\Psi}\|_{\cL(\Psi(X),\Psi(Y))}\\
 & = \sup\{ \|T\psi_l(D) f\|_{\Psi(Y)} : f\in \cS(X) \textrm{ with } \|f\|_{\Psi(X)} \leq 1, l\in \N_0 \}. 
\end{align*}
Moreover, one can check that for every $f\in F^{s,q}_\Psi(X)$ and $j\in \N_0$ we have that 
\begin{equation}
\|\psi_j(D) \cT f\|_{\Psi(Y)} \leq 3 \mu_w \| \psi_j(D) f\|_{\Psi(X)} 
\end{equation}

Fix $\Phi$ and $f\in \Phi(X) \cap F^{s,q}_{\Phi}(X)$, i.e. $Gf\in \Phi$, where
\[
Gf:= \left( \sum_{j\in \N_0} | 2^{sj} \psi_j(D) f(\cdot) |^q_X \right)^{1/q}. 
\]
For $h\in \Phi'$, $h\neq  0$, we set
\[
w:= w_{Gf, h, q}:= \cR(Gf)^{1-q} \cR' h.
\]
 Then $Gf\in L^q_w$, i.e. $f \in  F^{s,q}_{L^q_w}(X)$ and $\cT f := \cT_{L^q_w} f \in F^{s,q}_{L^q_w}(X)$; see \eqref{weight}.
Moreover, a~similar argument as in the proof of Theorem \ref{RdeF} gives 
\begin{align*}
3\mu_w  \|\cR\|_{\cL(\Phi)} \|Gf\|_\Phi \|\cR'\|_{\cL(\Phi')} \| h\|_{\Phi'}
%& \geq  \| \cT\|_{\cL(L^q_w(X),L^q_w(Y))}  \|\cR Gf\|_\Phi \|\cR'h\|_{\Phi'} \\
& \geq  3\mu_w \left( \int_\R \cR(Gf) \cR' h \, \ud t \right)^\frac{1}{q} \left(  \int_\R \cR(Gf) \cR' h \, \ud t \right)^\frac{1}{q'}\\
&\geq  \left( \int_\R G(\cT f)^q w \,\ud t \right)^\frac{1}{q} \left(  \int_\R \cR(Gf) \cR' h \ud t \right)^\frac{1}{q'}\\
% &\geq \int_{\R} G(\cT f) \cR' h \, \ud t\\
&\geq \int_\R G(\cT f) h \, \ud t
\end{align*}
Since 
\[
\sup \left\{ [w_{Gf,h,q}]_{A_q}: f\in \Phi(X)\cap F^{s,q}_{\Phi}(X), h\in \Phi' \right\} <\infty,
\]
\[
\mu:= \sup\left\{ 3 \mu_w: w= w_{Gf,h,q} \textrm{ with } f\in \Phi(X)\cap F^{s,q}_{\Phi}(X), h\in \Phi'\right\}<\infty
\] and, by the Lorentz-Luxemburg theorem, $\cT f\in F^{s,q}_{\Phi}(Y)$ with
\[
\| \cT f \|_{F^{s,q}_{\Phi}(Y)}\leq  \mu  \|\cR\|_{\cL(\Phi)} \|\cR'\|_{\cL(\Phi')} \| f \|_{F^{s,q}_{\Phi}(X)}.
\]
Therefore, since $ \Phi(X)\cap F^{s,q}_{\Phi}(X)$ is dense in $F^{s,q}_{\Phi}(X)$ (see Lemma \ref{completness}), $\cT$ extends to an operator in $\cL(F^{s,q}_{\Phi}(X), F^{s,q}_{\Phi}(Y))$. 
 Finally, by Lemma \ref{completness} and similar arguments to those presented in the case of Besov spaces above, one can show the consistency for operators resulting from different underlying Banach function spaces $\Phi$ and $\Psi$. This completes the proof of $(i)$.

$(ii)$ By Theorem \ref{RdeF}, $T$ extends to an operator in $\cL(\Phi(X), \Phi(Y))$. For instance, the last assertion of Corollary \ref{CZ bd} shows that $T$ satisfies the assumption of the part~$(i)$. Thus, the case when $\E = B^{s,q}_\Phi$ or $\E = F^{s,q}_\Phi$ is already proved above. 

For the consistency $\cT_\E$ with $\cT_{\widetilde\E}$, by Theorem \ref{RdeF}, it holds when $\E = \Phi_1$ and $\widetilde\E = \Phi_2$. By Lemma \ref{completness} and similar arguments to those presented in the proof of $(i)$, we easily get the general case.  
\end{proof}

\section{Multipliers and the regularity of their kernels}\label{sec multip}
In this section, we specify Theorem \ref{ext general} for a class of Fourier multipliers in which we are typically interested in the following sections. 

\subsection{The Fourier multipliers}
Let $X$ and $Y$ be Banach spaces.
 For  $a\in L^\infty(\cL(X,Y))$ we set  
\[
{D_a}: =\{ f\in \cS'(X): \cF f \in \cS'_r(X) \textrm{ and } a(\cdot)\cF f \in \cS'(Y)\} 
\] to denote the {\it initial} domain of a Fourier multiplier $a(D)$ associated with $a$ which is defined by 
\[ 
a(D)f := \cF^{-1} (a(\cdot) \cF f)\quad  (f\in D_a).  
\]
If $\cG (X)$ is a subspace of $\cS'(X)$ (equipped with the relative topology of $\cS'(X)$) and $a(D)$ has a unique extension to an operator in $\cL(\cG (X), \cS'(Y))$, then we denote such extension again by $a(D)$. Since $X\otimes \cF^{-1}\cD(\dot{\R})$ is a dense subset of $\cS'(X)$, if $X\otimes \cF^{-1}\cD \subset \cG(X)$, then such extension of $a(D)$ is unique. 

In particular, as was already mentioned in Section 1, if $a\in \cO_M(\cL(X,Y))$, 
then $a(D) \in \cL(\cS'(X), \cS'(Y))$ and 
\begin{equation}\label{dist repr}
a(D) f = \cF^{-1} (\Theta (a, \cF f)) \qquad f\in \cS'(X),
\end{equation}
where $\Theta$ denotes a hypocontinuous, bilinear map from $\cO_M(\cL(X,Y))\times \cS'(X)$ into $\cS'(Y)$ (see \cite[Theorem 2.1]{Am97}).  
The multipliers symbols, which are involved in the study of maximal regularity property of evolution equations usually are of $\cO_M$-class; see, e.g. Proposition \ref{OM class}. For instance, it is the case of the problem \eqref{P1} discussed in Section \ref{intro}, where the {\it solution} operator $U$ is given by a multiplier $a(D)$ with $a = (i\cdot + A)^{-1} \in \cO_M(\cL(X,Y))$ and $Y:=D_A$.

Further, by the convolution theorem (see e.g. \cite[Theorem 3.6]{Am97}), for all 
$a\in L^\infty(\cL(X,Y))$ and $f\in \cS(X)$, 
\begin{equation}\label{conv repr}
a(D)f = \cF^{-1}a \ast f.
\end{equation}

Here, $\cF^{-1} a \in \cS'(\cL(Y,X))$, that is, $\bracket{\cF^{-1}a}{\phi}:=\bracket{a}{\cF^{-1} \phi}$, $\phi \in \cS$. For more regular symbols $a$ such convolution representation of the corresponding multiplier $a(D)$ allows us to use the basic techniques of harmonic analysis to study of boundedness properties/invariant subspaces of $a(D)$.   
However, in general, such representation  makes sense 
only for $f$ in a proper subset of a domain of $a(D)$. For instance, even in the case 
$a$ in the $\cO_M$-class, in general, what one can say is that \eqref{conv repr} holds for $f \in \cE'(X) \cup S(X)$, where $\cE'(X)$ denotes the space of distributions with compact support. 

Depending on  support conditions or integrability conditions of $\cF^{-1}a$ this expression can be extended for a larger class of functions; see e.g. \cite[Remark 3.2 and Theorem 3.5]{Am97}.
It leads to a question on the consistency of $a(D)$ with its bounded extensions resulting from such convolution representation.
We address this question for a particular class of symbols below.

\subsection{The integral representation of $a(D)$.}\label{int represent}
For  symbols $a$, which arise in the study of abstract evolution equations, usually $\cF^{-1} a$ is a tempered distribution which is a regular one only on $\dot \R$.
 That is,  there exists a function $k\in L^1_{loc}(\dot \R;\cL(X,Y))$ such that for every $\phi \in \cD(\dot \R; \C)$
\[
\bracket{\cF^{-1}m}{\phi} = \bracket{k}{\phi}:= \int_\R k(t)\phi(t) \ud t.
\]
Let 
\[
Cf(t):= \int_\R k(r)f(t-r)\ud r 
\]
for every $f\in L^1_c(X)$ and a.e. $t\notin \supp f$.
Here, $L^1_c(X)$ stands for the space of all $X$-valued, measurable, Bochner integrable functions with compact support in $\R$. 
 By Young's inequality, this integral is absolutely convergent for a.e. $t\notin \supp f$. 

Note that, even in the case $a\in \cO_M(\cL(X,Y))$,  without any further information on $a(D)$ we only get that
\begin{equation}\label{si repr}
a(D)f(t) = Cf(t)
\end{equation}
for $f\in L^1_c\otimes X$ and a.e. $t\notin \supp f$ (a priori we know that $a(D)$ is only continuous on $\cS'(X)$;  cf. also \cite[Example 2.11]{HyWe06}). If we know that, e.g. $a(D)$ restricts to an operator in $\cL(L^p(X), L^p(Y))$ for some $p\in (1,\infty)$, then of course such representation holds for $f\in L^1_c(X)$.

In the case when, e.g. $X\otimes \cD$ is not dense in a Banach space $\E(X)$ under consideration, if there exists an operator $T$ in $\cL(\E(X), \E(Y))$ such that $Tf(t) = Cf(t)$ for every $f\in L^1_c(X)$ and a.e. $t\notin \supp f$, we do not have any {\it ad hoc} argument to show that $T$ is consistent with $a(D)$ on $\E(X)$, i.e., $T$ is a restriction of $a(D)$ to $\E(X)$. The Rubio de Francia iteration algorithm, Theorem \ref{RdeF}, makes it all rigorous.

The following result is a reformulation of \cite[Proposition 4.4.2, p.245]{St93}. 
 It shows how Mihlin type estimates of a symbol $a$ are reflected by its distributional kernel $\cF^{-1}a \in \cS'(\cL(X,Y))$. Recall that $\dot{\R}:=\R \setminus \{0\}$.  

\begin{lemma}\label{CZ cond}
Let $\gamma \in \N$ and $\gamma \geq 2$. If function $a\in \cC^\gamma(\dot\R ; \cL(X,Y))$ satisfies 
\[
[a]_{\m_\gamma}:=\max_{l=0,...,\gamma}\sup_{t\neq 0} \|t^l a^{(l)}(t)\|_{\cL(X,Y)}<\infty \quad \quad  {(\frak M_\gamma)} 
\]
then $\cF^{-1} a = k$ in $\cS'(\cL(X,Y))$ for a function $k\in \cC^{\gamma - 2}(\dot \R; \cL(X,Y))$ such that 
\[
[k]_{\cK_{\gamma-2}} := \max_{l=0,...,\gamma-2} \sup_{s\neq 0} \|t^{l+1} k^{(l)}(t)\|_{\cL(X,Y)}\leq C[a]_{\m_\gamma}, 
\]
where the constant $C$ is independent of $a$. 
\end{lemma}
The proof of this lemma follows the idea of the proof of its scalar counterpart given in Stein \cite[Proposition 4.4.2(a)]{St93} and is omitted here.
  
In the sequel, we say that $a$ satisfies the $\m_\gamma$-condition and write $a\in \m_\gamma (\cL(X,Y))$, if $a\in\cC^\gamma(\dot \R; \cL(X,Y))$ with $[a]_{\m_\gamma} <\infty$. 
Furthermore, let $\widetilde{\frak M}_\gamma(\cL(X,Y))$ denote a subclass of ${\frak M}_\gamma(\cL(X,Y))$, which consists symbols $a$ such that 
\[
[a]_{\widetilde\m_\gamma}:=\max_{l=0,...,\gamma}\sup_{t\neq 0} \|(1+|t|)^l a^{(l)}(t)\|_{\cL(X,Y)}<\infty \quad \quad  {(\widetilde{\frak M}_\gamma)} 
\]
and $a(0+)=a(0-)$ (equivalently, $a$ has continuous extension at $0$).
Moreover, if $a$ is such that $\cF^{-1}a \in \cC^1(\dot \R; \cL(X,Y))$ with $[\cF^{-1}a]_{\cK_1}<\infty$, then we say that $\cF^{-1}a$ satisfies the (standard) Calder\'{o}n-Zygmund conditions. 

Consequently, by \eqref{si repr}, if $a(D)$ is in $\cL(L^p(X), L^p(Y))$ and $\cF^{-1}a$ satisfies the Calder\'{o}n-Zygmund condition, then $a(D)$ is a Calder\'{o}n-Zygmund operator according to the definition from Section \ref{sec extension}.
The boundedness properties of such operators are the subject of the rest of this section.

\subsection{The boundedness of Fourier multipliers}

Let $\{\psi_j\}_{j\in\N_0}$  be the resolution of the identity on $\R$; see Section \ref{spaces}.  
The proof of Theorem \ref{ext general}  shows that the study of the boundedness of a multiplier $a(D)$ on vector-valued Besov type spaces $B^{s,q}_\E$, defined on the basis of Banach space $\E$, reduces to the study of the uniform boundedness of its dyadic parts $(\psi_j a)(D)$, $j\in \N_0$, on the underlying space $\E(X)$, that is, 
\begin{equation}\label{dyiadic parts}
\sup_{j\in \N_0} \|  (\psi_j a)(D)  \|_{\cL(\E(X),\E(Y))}<\infty.
\end{equation}

In \cite[Proposition 4.5]{Am97} Amann showed that if $a\in \widetilde \m_2(\cL(X,Y))$, then, in particular, the condition \eqref{dyiadic parts} holds for $\E \in \{BUC, \cC_0, L^p; 1\leq p\leq \infty\}$. See also \cite{HyWe06}(and the references therein) for a systematic study of the conditions which imply the boundedness of $a(D)$ on the classical vector-valued Besov spaces (corresponding to $L^p$ spaces).

The following result extends \cite[Proposition 4.5]{Am97} for a larger class of spaces $\E$. We point out that for symbols which arise from particular types of evolution equations, the $\widetilde\m_1$-condition (respectively, the $\m_1$-condition) implies the $\widetilde\m_\gamma$-condition (respectively, the $\m_\gamma$-condition) for every $\gamma \in \N$; see, e.g. Propositions \ref{main 2} and \ref{OM class}.

\begin{proposition}\label{bd multip}
Let $\{\psi_j\}_{j\in \N_0}$ be a dyadic resolution of the identity on $\R$. Let $\Phi$ be a Banach function space such that the Hardy-Littlewood operator $M$ is bounded on $\Phi$ and its dual $\Phi'$.
 Then the following assertions hold.
\begin{itemize}
\item [(i)] 
Assume that $a \in {\widetilde\m}_2 (\cL(X,Y))$. Let $s\in \R$ and $q\in [1,\infty]$. Then the operator $a(D)$ extends to a linear operator $\cT_\Phi$ in $\cL(B^{s,q}_\Phi (X), B^{s,q}_\Phi (Y))$. If $\Psi$ is another Banach function space such that $M$ is bounded on $\Psi$ and $\Psi'$, then the corresponding operator $\cT_\Psi$ is consistent with  $\cT_\Phi$. Moreover, if $q=\infty$, then $\cT_\Phi$ is $\sigma(B^{s,\infty}_\Phi(X), B^{-s,1}_{\Phi'}(X^*))$-to-$\sigma(B^{s,\infty}_\Phi(Y), B^{-s,1}_{\Phi'}(Y^*))$-continuous.

In addition, if $a(D)$ has the extension to an operator in $\cL(\cG(X), \cS'(Y))$, where $\cG(X) \subset \cS'(X)$ with $B^{s,q}_\Phi(X)\subset \cG(X)$, then $a(D)f = \cT_\Phi f$ for every $f\in B^{s,q}_\Phi(X)$. In particular,  $a(D)$ restricts to an operator in  $\cL(B^{s,q}_\Phi (X), B^{s,q}_\Phi (Y))$.

In the case when $s\in \R$ and $q\in (1,\infty)$, the above statements remain true if we replace the Besov space $B^{s,q}_\Phi$ with the Triebel-Lizorkin space $F^{s,q}_\Phi$. 

\item [(ii)]  Assume that  $a \in {\m}_3 (\cL(X,Y))$ and $a(D) \in \cL(L^p(X), L^p(Y))$ for some $p\in (1,\infty)$. Then for every 
\[
\E \in \left\{\,\Phi,\,B^{s,q}_\Phi,\, F^{s,r}_\Phi: s\in \R, q\in [1,\infty], r\in (1,\infty) \right\} 
\]
the operator $a(D)$ extends to an operator $\cT_\E$ in $\cL(\E(X), \E(Y))$.
For $\E = B^{s,\infty}_\Phi$, $s\in \R$, the corresponding operator $\cT_\E$ is $\sigma(B^{s,\infty}_\Phi(X), B^{-s,1}_{\Phi'}(X^*))$-to-$\sigma(B^{s,\infty}_\Phi(Y), B^{-s,1}_{\Phi'}(Y^*))$-continuous. 

Moreover, if $a(D)$ has the extension to an operator in $\cL(\cG(X), \cS'(Y))$, where $\cG(X)$ is a subspace of $\cS'(X)$ such that $\E(X)\subset \cG(X)$, then $a(D)f = \cT_\E f$ for every $f\in \E(X)$. 
\end{itemize} 
\end{proposition}

\begin{proof} 
We prove that the operator $T:=a(D)_{|\cS(X)}$ satisfies the assumptions of Theorem \ref{ext general}$(i)$, i.e. we show that for every $\cW\in A_2(\R)$ with $\sup_{w\in \cW} [w]_{A_2} <\infty$ 
\begin{equation}\label{RdeF est R}
\sup_{w\in \cW} \sup_{j\in \N_0} \| (\phi_j a) (D)\|_{\cL(L^2_w(\R; X), L^2_w(\R; Y))} < \infty. 
\end{equation} 
First note that for each $j\geq 1$,
\[
\cF^{-1}(\psi_j a) = 2^j \cF^{-1}(\eta a(2^j\cdot))(2^j\cdot), 
\] where $\eta (t):=\psi(t)- \psi(2t)$ $(t\in \R)$. By Leibniz' rule we get that there exists a constant $C$ such that for every $l\in \{0,1,2\}$, $j\in \N$ and $t\in \R$,
\[
\left\|\left( \eta(t) a(2^j t) \right)^{(l)}\right\|_{\cL(Y,X)} 
%\leq C \chi_{I}(t) \max_{k=0,1,2} \sup_{2^{j-1}\leq |s| \leq 2^{j+1}} \|a^{(k)}(s)\|_{\cL(X,Y)} 
\leq C [a]_{\m_2} \chi_{I}(t), 
\]
where $\chi_I$ stands for the characteristic function of the set $I:=\{\frac{1}{2}\leq |t| \leq 2 \}$. Therefore, for every $l\in \{0,1,2\}$ and $j\in \N$ we get that 
\[
\|(\eta\, a(2^j\cdot))^{(l)}\|_{L^1(\R; \cL(X,Y))}\leq 4 C [a]_{\m_2}. 
\] 
Since for every $l\in \{0,1,2\}$, $j\in \N$ and $t\in \R$
\[
\left\|\cF^{-1}\left( (\eta\, a(2^{j}\cdot))^{(l)} \right)(t)\right\|_{\cL(X,Y)} = \left\| t^l \cF^{-1}\left( \eta a(2^{j}\cdot) \right)\right\|_{\cL(X,Y)} \leq 4C [a]_{\m_2}, 
\]
\[
 \left\|\cF^{-1}\left( \eta\, a(2^{j}\cdot) \right)(t)\right\|_{\cL(X,Y)} \leq 8 C [a]_{\m_2} \frac{1}{1+t^2}. 
\]
Consequently, 
\[
 \left\|\cF^{-1}(\phi_j a)(t)\right\|_{\cL(X,Y)} \leq 8 C [a]_{\m_2} \frac{2^j}{1+2^{2j}t^2} \qquad  (t\in \R, j\in \N).
\]
Note that $h_j(t):=\frac{2^j}{1+2^{2j}t^2}$ $(t\in \R)$, $j\in \N$, are even, (radially) decreasing function such that $\|h_j\|_{L^1(\R)} = \pi$. By \cite[Chapter II, (17) p.57]{St93}, for each $f\in \cS(X)$ and $t\in \R$ we have that 
\[
\sup_{j\in \N} \left| (\cF^{-1}(\phi_j a)\ast f)(t) \right|_Y \leq 8C[a]_{\m_2}  \sup_{j\in \N} (h_j \ast |f|_{X}) (t) \leq 8\pi C[a]_{\m_2}  M_\R (|f|_X)(t). 
\]

For $j=0$, note that if $a\in \widetilde\m^2(\R;\cL(X,Y))$, then 
\[
\partial (\psi a) = (\psi a)' + [a(0+) - a(0-)] \delta_0 = (\psi a)'
\]
and 
 \[
\partial^2 (\psi a) = (\psi a)'' + [a'(0+) - a'(0-)] \delta_0. 
\]
Here, $\delta_0$ stands for Dirac's measure at $0$. Hence, respectively, taking  Fourier's transform, we have that
\[
i t \cF^{-1}(\psi a)(t) = \cF^{-1}(\partial^2 (\psi a)) (t) = \cF^{-1}\left((\psi a)'\right)
\]
\[
- t^2 \cF^{-1}(\psi a)(t) = \cF^{-1}(\partial (\psi a)) (t) = \cF^{-1}\left((\psi a)''\right) (t) + [a'(0+) - a'(0-)] \qquad (t\in \R). 
\]
Thus, applying arguments which were presented above for the function $\psi_j a$ $(j\in \N)$ to the function $\psi m$, we infer that for some constant $C$ and every $f\in \cS(X)$
\[
|\cF^{-1}(\psi a)\ast f(t)|_X \leq C [a]_{\widetilde \m_2} M(|f|_X)(t) \qquad (t\in \R). 
\] 
Therefore, Muckenhoupt's theorem (see also \cite{Bu93}) gives \eqref{RdeF est R}. 

By Theorem \ref{ext general}$(i)$, $T$ extends to an operator $\cT_\Phi$ in $\cL(B^{s,q}_\Phi(X), B^{s,q}_\Phi(X))$.  
The statement on the continuity of $\cT_\Phi$ if $q=\infty$ follows from the fact that for the adjoint operator of $\cT_{l,\Phi}\in \cL(\Phi(X), \Phi(Y))$, $\cT_{l,\Phi}^*  g \in \Phi'(X^*)$ for all $g\in \Phi'(Y^*)$. Indeed, one can easily check that $\cT_{l,\Phi}^*$ agrees on $\Phi'(Y^*)\subset [\Phi(Y)]^*$ with the extension of the operator $(a^*\psi_l)(D)$ to an operator in $\cL(\Phi'(Y^*), \Phi'(X^*))$, which is obtained by means of Theorem \ref{ext general}. Here,  $a^*(t):= a(t)^* \in \cL(Y^*, X^*)$, $t\neq 0$. It proves the first statement of $(i)$.

For the second one, assume that $a(D)$ possesses a continuous extension on a subset $\cG(X)\subset \cS'(X)$ such that $B^{s,q}_\Phi(X)\subset \cG(X)$. By Lemma \ref{completness} and \eqref{resol on R}, $B^{s,q}_\Phi (X) \hookrightarrow \cS'(X)$ and for every $f\in B^{s,q}_\Phi(X)$, $\sum_{j=0}^N \psi_j(D) f \rightarrow f$ as $N\rightarrow \infty$ both in $\cS'(X)$ and $B^{s,q}_\Phi(X)$ (when $q=\infty$ the convergence holds in $\sigma(B^{s,\infty}_{\Phi}(X), B^{-s,1}_{\Phi'}(X^*))$-topology of $B^{s,\infty}_{\Phi}(X)$), we get $a(D)f=\cT_\Phi f$ for all $f\in B^{s,q}_\Phi(X)$.

The proof of the last statement of $(i)$ regarding the Triebel-Lizorkin spaces is based on the corresponding statement of Theorem \ref{ext general}$(i)$ and mimics arguments presented above. This completes the proof of the part $(i)$. 

For the proof of $(ii)$, first note that $a(D)$ is a Calder\'{o}n-Zygmund operator; see  Lemma \ref{CZ cond} and considerations made in Subsection \ref{int represent}. Thus, the proof of this part relies on Theorem \ref{ext general}$(ii)$ and uses similar arguments to those presented above. 
\end{proof}

\begin{remark}\label{R bound}
(a) The additional statements of the points $(i)$ and $(ii)$ of Proposition \ref{bd multip} are mainly address to the case when  $a\in \cO_M(\cL(X,Y))$. Then, $a(D)$ has a {\it canonical extension} to an operator in $\cL(\cS'(X), \cS'(Y))$; see, e.g. Proposition \ref{OM class} (cf. also \cite[Problems 3.2 and 3.3]{HyWe06}).

(b) We conclude this section with the condition that guarantees a multiplier $a(D)$ is in $\cL(L^p(X), L^p(Y))$ for some $p\in (1,\infty)$. We refer the reader, e.g. to \cite[Section 1]{HyWe07} for a short survey on the theory of operator-valued multipliers on $L^p$ spaces, where further references to seminal papers in this area, including \cite{HyWe07}, can be found. 
Here, we just mentioned Weis' result \cite{We01} who first generalized the Mihlin theorem for operator-valued symbols by requiring $R$-boundedness instead of norm boundedness in Mihlin's condition $(\m_1)$.
Recall that, if $X$ and $Y$ are $U\!M\!D$ Banach spaces and a symbol $a\in L^\infty(\cL(X,Y))\cap \cC^1(\dot \R;\cL(X,Y))$ satisfies  
\[
R(\{ a(t), t a'(t): t\neq 0 \}) <\infty, \qquad \quad (R\m_1)
\]
then $a(D)\in\cL(L^p(X), L^p(Y))$ for every $p\in (1,\infty)$. 
We refer the reader, e.g.  to \cite{KuWe04} or \cite{DeHiPr03a}, for the background on $R$-boundedness and $U\!M\!D$-spaces.

Recently, in \cite{FaHyLi20} was proved that the same conditions implies also that $a(D)\in \cL(L^p_w(X), L^p_w(Y))$ for every $w\in A_p$ and every $p\in (1,\infty)$; see  
\cite[Theorem 3.5(a)]{FaHyLi20}. Moreover, it is readily seen from the proof that for every $\cW\subset A_p$ with $\sup_{w\in \cW} [w]_{A_p}<\infty$ 
\[
\sup_{w\in \cW}\|a(D)\|_{\cL(L^p_w(X), L^p_w(Y))} < \infty.
\]
Therefore, $T:=a(D)_{|\cS(X)}$ satisfies the assumptions of Theorem \ref{ext general}$(ii)$. Consequently, if $X$ and $Y$ have $U\!M\!D$ property, then it allows us to relax the assumptions of Theorem \ref{bd multip}$(ii)$ to get the same conclusion stated therein. 

(c)  For multiplier symbols arising in the study of evolution equations (e.g., see \eqref{P no c} below), the $R$-boundedness of their ranges usually implies that they satisfy the $R\m_1$-condition stated above; see the last statement of Proposition \ref{main 2} and Proposition \ref{OM class}. Recall that, by Cl\'ement and Pr\"uss \cite{ClPr01}, the $R$-boundedness of the range of a symbol $a$ is also the necessary condition for $a(D)$ to be in $\cL(L^p(X), L^p(Y))$. Therefore,  if the underlying Banach space $X$ has $U\!M\!D$ property, then it leads to the characterisation of the $L^p$-maximal regularity of such equations in the term of $R$-boundedness; see, e.g. \cite[Theorem 4.1]{ApKe20}. 
\end{remark}

\section{Maximal regularity of integro-differential equations}\label{app}

In this section, we illustrate our preceding general results by an abstract integro-differential equation with a general convolution term; see \eqref{P} below.
We discuss the particular forms of this equation in the following section. 

We start with a preliminary observation. Let $A$, $B$ and $P$ denote densely defined, closed, linear operators on a Banach space $X$.  
Set $D_A$, $D_B$ and $D_P$ for the domains of these operators equipped with the corresponding graph norms. Since $A\in \cL(D_A, X)$, the {\it evaluation} of $A$ on $\cS'(D_A)$ gives an operator $\cA$ in $\cL(\cS'(D_A), \cS'(X))$, i.e.,  $(\cA u)(\phi):=A(u(\phi))$ for every $u\in \cS'(D_A)$ and $\phi\in \cS$. The symbols $\cB$ and $\cP$ have the analogous meaning below. Note that the distributional derivative $\partial$ commutes with such extensions, i.e. $\partial \cA u  = \cA \partial u$ for all $u\in \cS'(D_A)$, and similarly for $\cB$ and $\cP$. 

Moreover, let $c\in \cS'(\cL(Z,X))$, where $Z\subset X$ denotes a Banach space.  
Consider the following problem 
\begin{equation}\label{P}
 \partial \cP \partial u + \cB \partial u + \cA u + c\ast u = f \quad \quad \textrm{ in } \cS'(X).
\end{equation}
For a given $f\in \cS'(X)$, by a {\it distributional solution} of \eqref{P} we mean a distribution $u\in \cS'(D_A)$ such that $\partial u \in \cS'(D_B) \cap \cS'(D_P)$ and the convolution of $c$ with $u$ can be (at least formally) well-defined in $\cS'(X)$.  
Below we assume that $c$ is the inverse Fourier transform of a given function $\hat c$ in $L^\infty(\cL(Z,X))$, i.e. $c = \cF^{-1}\hat c$, and define $c\ast u$ by $\hat c(D) u$. 

Moreover, we restrict our study of \eqref{P} to the case when $D_A \hookrightarrow  D_B, D_P$ and $Z= D_A$. The more general case of the main results of this section, when $Z\neq D_A$, can be proved easily following the presented proofs. Let $Y:= D_A$. Since $A,B,P \in \cL(Y,X)$, we get that  $\cA, \cB, \cP \in \cL(\cS'(Y), \cS'(X))$. Since $\partial \cP \partial u = \partial(\partial \cP u)$ and $\partial \cB u = \cB \partial u$ for all $u\in \cS'(Y)$, the solvability of \eqref{P} in the sense of distributional solutions is equivalent to that of
\begin{equation}\label{PP}
 \partial(\partial  \cP u) +  \partial \cB u + \cA u + c\ast u = f \quad \quad \textrm{ in } \cS'(X).
\end{equation}
Of course, it is not the case for the study of the well-posedness of \eqref{P} and \eqref{PP} in the sense of strong solutions. It is emphasized in the proof of Theorem \ref{main for itegro-dff}; see also Lemma \ref{strong solution} below. We call a function $u$ the {\it strong solution} of \eqref{P} if 
\[
u\in L^{1}_{loc}(D_A),\quad c\ast u \in L^1_{loc}(X),\quad u\in W^{1,1}_{loc}(D_B\cap D_P),\quad Pu'\in W^{1,1}_{loc}(X)
\] and  
 \begin{equation}\label{P strong}
( P  u')'(t) +  B u'(t) + A u(t) + (c\ast u)(t) = f(t) \quad \quad \textrm{ a.e. }t\in \R.
\end{equation}
Equivalently, $u$ is a strong solution if $u\in W^{1,1}_{loc}(X)$ and each summand on the left side of \eqref{P} is a regular distribution (cf. Lemma \ref{strong solution} below).

We start with an auxiliary result on the regularity of the kernels of multiplier operators involved in the study of \eqref{P}. 

\begin{proposition}\label{main 2}
Let $A, B, P$ and $c$ be as above. 
Assume that for every $t\in \R\setminus \{0\}$ the operator
\begin{equation}\label{fun c}
b(t):= -t^2 P + it B + A + \hat c(t) \in \cL(Y, X).
\end{equation}
is invertible. Set 
\begin{equation}\label{functions}
a(t) := b(t)^{-1}, \quad a_0(t):= t B a(t),\quad a_1(t):=t^2 P a(t) \quad (t\neq 0).
\end{equation}

Suppose that $a\in L^\infty(\cL(X,Y))$ and $a_j\in L^\infty (\cL(X))$ for $j=0,1$. 

Then, for every $\gamma \in \N$, if $\hat c \in \widetilde{\frak{M}}_\gamma (\cL(Y,X))$ (respectively, $\hat c \in {\frak{M}}_\gamma (\cL(Y,X))$), then $a \in \widetilde{\frak{M}}_\gamma(\cL(X,Y))$  and  $a_0, a_1 \in \widetilde{\frak{M}}_\gamma(\cL(X))$)(respectively, $a \in {\frak{M}}_\gamma(\cL(X,Y))$ and  $a_0, a_1 \in {\frak{M}}_\gamma(\cL(X))$).

In addition, if the ranges of the functions $a$ and $a_0, a_1$ are $R$-bounded, then if $\hat c$ satisfies the $R\m_1$-condition, then $a$ and $a_0$, $a_1$ do.

\end{proposition}

\begin{proof} Note that $b \in \cC^\gamma (\dot \R; \cL(Y,X))$ and, since the inversion map is analytic, $a\in \cC^\gamma (\dot \R; \cL(X,Y))$. Moreover, $a$ has a continuous extension at $0$ if $\hat c$ does.
Before we proceed by induction, we need to consider the cases when $\gamma = 1,2,3$ separately.

In the case when $\gamma =1$, since $b(t)a(t) = I$, by Leibniz' rule, for every $t\in \dot \R$ we get 
\begin{equation}\label{est 1}
-a'(t) = a(t) b'(t) a(t) \qquad (\textrm{in } \cL(X,Y)),
\end{equation}
 where
\[ b'(t) = -2t P + iB + \hat c '(t). \]
By our assumptions we get $ b'a, (\cdot)b'a \in L^\infty(\cL(X))$, which shows that  $a \in \widetilde\m_1(\cL(X,Y))$. 

Similarly,  when $\gamma = 2$, we get 
\begin{equation}\label{est 2}
-a''(t) = a'(t) b'(t) a(t) + a(t) b''(t) a(t) + a(t) b'(t) a'(t),
\end{equation}
where 
\[
b''(t) = - 2P + \hat c'' (t). 
\]
Therefore, our assumptions yield directly that $b''a, (\cdot)^2 b'' a \in L^\infty(\cL(X))$ and, by the step for $\gamma = 1$, $b'a' \in L^\infty(\cL(X))$. Consequently,    
$a \in \widetilde\m_2(\cL(X,Y))$.

For $\gamma =3$, note that  
\begin{align}\label{est 3}
-a'''(t)  = & a''(t) b'(t) a(t) + a'(t) b''(t) a(t) + a'(t) b'(t) a'(t)\\
\nonumber & + a'(t) b''(t) a(t) + a(t) b'''(t) a(t) + a(t) b''(t) a'(t) \\
\nonumber & +  a'(t) b'(t) a'(t) +  a(t) b''(t) a(t) +  a(t) b'(t) a''(t),    
\end{align}
where
\[
b'''(t) = \hat c '''(t).
\]
Therefore, the steps for $\gamma =1$ and $\gamma =2$ show that $a \in \widetilde\m_3(\cL(X,Y))$.

Now we can proceed by induction for $\gamma > 3$. Fix $3\leq \alpha \leq \gamma$. Then 
\[
a^{(\alpha)} = \sum_{0 \leq k,l,m  \leq \alpha,\, l\neq 0 \atop k+l+m = \alpha }
a^{(k)}b^{(l)}a^{(m)} \qquad (\textrm{on } \dot \R).
\]
By the induction step, for every $k=0, ..., \alpha - 1$, the function $(\cdot)^k a^{(k)}$ is in $L^\infty(\cL(X,Y))$. Moreover, in the case when $k\neq 0$, since $l+m\leq \alpha -1$, by the induction, we have that the functions $(\cdot)^{l+m}b^{(l)}a^{(m)}\in L^\infty(\cL(X))$. 
Therefore, consider the case of $k=0$, that is, $l+m = \alpha$. If $l=\alpha$, then 
$b^{(\alpha)}=\hat c^{(\alpha)}$, thus $(\cdot)^{\alpha} b^{(\alpha)}a \in L^\infty(\cL(X))$. 
In the case $l < \alpha$, i.e., $1\leq m \leq \alpha-1$, we have
\[
b^{(l)} a^{(m)} = \sum_{0 \leq m_1,m_2,m_3  \leq m, \, m_2\neq 0 \atop  m_1+m_2+m_3= m }
b^{(l)} a^{(m_1)}b^{(m_2)}a^{(m_3)} \qquad (\textrm{on } \dot \R).
\]
Since $l+m_1\leq l+m-m_2\leq \alpha-1$ and $m_2+m_3\leq m \leq \alpha-1$, by the induction step, we get $(\cdot)^{l+m}b^{(l)} a^{(m)} \in L^\infty(\cL(X))$. It shows that $a\in \widetilde\m_\gamma(X,Y)$. 
Applying the formulas obtained for the function $a$, one can easily get desired statement for $a_0$ and $a_1$. We omit details.

The proof in the case of the ${\frak{M}}_\gamma$-condition follows the one given above.

For the last statement, recall that directly from the definition of $R$-boundedness,  for any Banach spaces $X,Y, Z$, if $\tau , \sigma \subset \cL(X,Y)$ and $\rho \subset \cL(Y,Z)$ are $R$-bounded, then the families $\tau + \sigma$ and $\tau \circ \rho$ are $R$-bounded (see e.g. \cite[Fact 2.8, p. 88]{KuWe04}). Therefore, 
the proof of the last statement mimics exactly the argument given already for the usual, norm boundedness.
 
\end{proof}

 The following lemmas allow to derive the strong solutions of the problem \eqref{P}.
 
\begin{lemma}\label{strong solution} 
Let $X$ and $Y$ be Banach spaces such that $Y\hookrightarrow X$. Let $C$ be a closed, linear operator on $X$ such that $Y \hookrightarrow D_C$. 

Let $u \in L^1_{loc}(Y)$ be such that $\partial u$ and $\partial Cu$ are regular $X$-valued distribution, i.e. $\partial u, \partial Cu \in L^1_{loc}(X)$. 

Then,  for almost all $t\in \R$ the strong derivative $u'(t)$ of $u$ at $t$ exists in $X$, $u'(t) \in D_C$ and $Cu'(t) = (\partial C u)(t)$. 
\end{lemma} 
The proof of this lemma follows easily from the closedness of $C$, which yields
\begin{align*}
\int_\R^X (\partial Cu)(t) \phi(t) \ud t & =  (\partial Cu)(\phi) =  (\cC \partial u)(\phi) =  - C \left( \int^Y_\R u(t) \phi'(t) \ud t\right) =\\
& = C \left( \int^X_\R u'(t) \phi(t) \ud t\right) \quad \textrm{for all }\phi \in \cS,
\end{align*}
and Lebesgue's differential theorem.  Here, we use the symbol $\int^X_\R$ to indicate in  which space the integration over $\R$ is considered. Note that $\int^X_\R u'(t) \phi(t) \ud t \in Y$ for all $\phi \in \cS$. 

The following result is a counterpart of the characterisation of the classical Besov spaces ($X=\C$, $\Phi= L^p$) in terms of distributional derivatives; see \cite[Theorem 2.3.8]{Tr83}. 

\begin{lemma}\label{lem equiv norms} Let $s\in \R$, $q\in [1,\infty]$ and let $\Phi$ denote a Banach function space such that the Hardy-Littlewood operator $M$ is bounded on $\Phi$ and its dual $\Phi'$. 
Then, for every distribution $f\in \cS'(X)$, $f$ belongs to $B^{s,q}_{\Phi}(X)$ if and only if $f$ and $\partial f$ belong to $B^{s-1,q}_\Phi(X)$. Moreover, the function 
\begin{equation}\label{equiv norm}
B^{s,q}_\Phi(X) \ni f \mapsto \left\| f \right \|_{B^{s-1,q}_\Phi(X)} + \left\| \partial^\alpha f \right \|_{B^{s-1,q}_\Phi(X)}
\end{equation}
is an equivalent norm on $B^{s,q}_\Phi(X)$.
\end{lemma}
The proof of Lemma \ref{lem equiv norms} reproduces the ideas of the proof of the classical case when $X=\C$ and $\Phi=L^p$ from \cite{Tr83}.  Since the proof, even in the classical case, is somewhat complex and based on several ingredients in which formulation the characteristics $p$ and $q$ are involved, we sketch the line of this extension and underline the main supplementary observation should be made.

\begin{proof}[Proof of Lemma \ref{lem equiv norms}]
First, we show that the lifting operator $\cJ$ given by 
\[
\cJ f = \cF^{-1} (1+(\cdot)^2)^{1/2} \cF f \quad (f\in \cS(X))
\]
maps $B^{s,q}_\Phi(X)$ isomorphically onto $B^{s-1,q}_{\Phi}(X)$. 

Let $\{\psi_j\}_{j\in N_{0}}$ be the resolution of the identity on $\R$ as in Section \ref{spaces}. Set 
\[
\phi_j(t):=2^{-j}(1+t^2)^{1/2} \psi_j(t) \quad (j\in \N_0). 
\] 
 In order to show that $\cJ f\in B^{s-1,q}_\Phi(X)$ for every $f\in B^{s,q}_\Phi(X)$ it is sufficient to show that there exists a constant $\mu$ such that for all $j\in \N_0$
\[
\left\| 2^{(s-1)j} \psi_j(D) \cJ f  \right\|_\Phi(X) \leq  \mu \left\| 2^{sj} \psi_j(D)f   \right\|_{\Phi(X)}.
\] 
For note that for every $f\in \cS'(X)$ and every $t\in \R$
\begin{align*}
| 2^{(s-1)j} \psi_j(D) \cJ f (t) |_X & = | 2^{js}  \phi_j (D) f(t) |_X \\
& \leq 2^{js} \sup_{r\in \R} \frac{|\phi_j(D) f (t - r) |_X}{1+ |2^{j} r|} =: 2^{sj} \phi^\ast_j f(t). 
\end{align*} 
We show that there exists a constant $\mu$ such that for every $f\in B^{s,q}_\Phi(X)$ and every $j\in \N_0$ we have that
\begin{equation}\label{max ineq.}
\| \phi^\ast_jf\|_\Phi \leq \mu \| \phi_j(D) f\|_{\Phi(X)} \quad \textrm{ and } \quad \| \phi_j(D) f\|_{\Phi(X)} \leq \mu \| \psi_j(D) f\|_{\Phi(X)}. 
\end{equation}
The second inequality above follows from Proposition \ref{bd multip} since 
\[
\phi_j(D) f = \cF^{-1} 2^{-j} (1+(\cdot)^2)^{1/2} \chi_j \cF \cF^{-1} \psi_j \cF f
\]
and, as it is readily seen, the functions $\rho_j: \R \ni t \mapsto 2^{-j} (1+t^2)^{1/2} \chi_j(t)$, $j\in \N_0$, satisfy the  $\widetilde \m_3$-condition uniformly in $j$. Therefore, 
$\rho_j(D)$, $j\in N_0$, restrict to uniformly bounded operators in $\cL(B^{s,q}_\Phi(X))$.  

To get the first inequality in \eqref{max ineq.}, by the same argument as in the proof of \cite[Theorem 1.3.1]{Tr83}, we obtain the existence of a constant $\mu$ such that for every $g\in \cS(X)$ and $j\in N_0$
\[
\phi^\ast_j g(t) \leq \mu M\left(|\phi_j(D)g|_X \right)(t).
\]
Now by an approximation argument similar to that given in the proof of \cite[Theorem 1.4.1]{Tr83} one can show that for every Muckenhoupt weight $w\in A_2$ for every $f\in L^2_w(X)$ we have that 
\[
\|\phi^\ast_j f\|_{L^2_w} \leq \mu \| M\|_{L^2_w} \| \phi_j(D) f\|_{L^2_w(X)}.
\]

Since $\|M\|_{L^2_w}\lesssim [w]^2_{A_2}$ (see \cite[Theorem 2.5]{Bu93}), by an argument similar to that from the proof of Lemma \ref{completness}(the case $\E = F^{s,q}_\Phi$), we get the first inequality in \eqref{max ineq.}. 
Therefore, $\cJ$ maps $B^{s,q}_\Phi(X)$ into $B^{s-1,q}_\Phi(X)$. 
Note that $\cJ$ is an isomorphism onto $\cS'(X)$, and by the same argument to that above one can show that $\cJ^{-1} f = \cF^{-1}(1+(\cdot)^2)^{-1/2}\cF f \in B^{s,q}_\Phi(X)$ for each $f \in B^{s-1,q}_\Phi(X)$. Therefore, the lifting property of $\cJ$ is proved. 

Now fix $f\in B^{s,q}_\Phi(X)$. We show that $f, \partial f\in B^{s-1,q}_\Phi(X)$. 
Obviously, $f\in B^{s-1,q}_\Phi(X)$. For $\partial f$, note first that the function $\rho(t) := t(1+t^2)^{-1/2}$, $t\in \R$, satisfies the $\widetilde \m_3$-condition and is of $\cO_M$-class. Thus, by Proposition \ref{bd multip}, $\rho(D)$ restricts to an  operator in $\cL(B^{s-1,q}_\Phi(X))$, and 
\begin{equation}
\| \partial f \|_{B^{s-1,q}_\Phi(X)}  = \|\rho(D) \cJ f\|_{B^{s-1,q}_\Phi(X)} \leq 
\| \rho(D) \cJ \|_{\cL(B^{s-1,q}_\Phi(X))} \|f\|_{{B^{s,q}_\Phi(X)}} 
\end{equation}
Finally, let $\eta(t) :=  ((1+t^2)^{1/2} - 1)t^{-1}$. It is readily seen that $\eta$ satisfies the $\widetilde \m_3$-condition. Therefore, $\eta(D)$ restricts to an operator in $\cL(B^{s,q}_\Phi(X))$ and for every $f\in \cS'(X)$ 
\[
\cJ f = f + \eta(D) \partial f. 
\]
Since $\cJ^{-1}$ maps isomorphically $\B^{s-1,q}_\Phi(X)$ onto $\B^{s,q}_\Phi(X)$, if  $f \in B^{s-1,q}_\Phi(X)$ with $\partial f\in B^{s-1,q}_\Phi(X)$, then $f \in B^{s,q}_\Phi(X)$. It completes the proof.

\end{proof}

\begin{theorem}\label{main for itegro-dff}
Let  $A, B$ and $P$ be closed, linear operator on a Banach space $X$ such that $D_A \hookrightarrow D_B, D_P$. Let $c\in \cS'(\cL(Y,X))$ with $\hat c = \cF c \in L^\infty(\cL(Y,X))$. 
Let $\Phi$ be a Banach function space over $(\R, \ud t)$ such that the Hardy-Littlewood  operator $M$ is bounded on $\Phi$ and its dual $\Phi'$.

\begin{itemize}
\item[(i)] Assume that the functions $b, a , a_0, a_1$ satisfy the assumptions of Proposition \ref{main 2}.  
If $\hat c \in \widetilde{\m}_2(\cL(Y,X))$, then for every $s\in\R$, $q\in [1,\infty]$
the problem $\eqref{P}$ has $B^{s,q}_\Phi$-maximal regularity property. That is, for every $f\in B^{s,q}_\Phi (X)$ the problem \eqref{P} has a unique distributional solution $u$ in $B^{s,q}_\Phi (Y)$ such that 
\[
\partial \cP \partial u,\, \cB \partial u,\,  \cA u,\, c\ast u \in B^{s,q}_\Phi (X). 
\] 
In addition, if $(\cdot)a(\cdot) \in L^\infty(\cL(X))$, then 
\[
\partial u,\, \cB \partial u,\, \cP \partial u  \in \B^{s,q}_\Phi(X),
\]
and, in the case when $s>0$, the function $u$ is a strong solution of \eqref{P}, $\partial u = u'$ and $(Pu')' \in \B^{s,q}_{\Phi}(X)$.

\item [(ii)] Assume that for every $t\neq 0$
\[
b(t)= - t^2 P + it B + A + \hat c(t) \in \cL(Y,X)
\] is invertible and let $\hat c \in \m_3(\cL(Y,X))$. 
If \eqref{P} has $L^p$-maximal regularity property, 
 then for every
\[
\E \in \left\{\,\Phi,\, B^{s,q}_\Phi,\, F^{s,r}_\Phi: s\in \R, q\in [1,\infty], r\in (1,\infty) \right\} 
\] it has $\E$-maximal regularity property.
  
In addition, if $(\cdot)a(\cdot) \in L^\infty(\cL(X))$, then 
\[
\partial u,\,  \cB \partial u,\, \cP \partial u  \in \E(X).
\]
Furthermore, if $\E\subset L^1_{loc}$, then the function $u$ is a strong solution of \eqref{P}, $\partial u = u'$ and $(Pu')' \in \E(X)$.

In particular, the last statement holds for all 
\[
\E \in \left\{\,\Phi,\, B^{s,q}_\Phi,\, F^{s,r}_\Phi: s>0, q\in [1,\infty], r\in (1,\infty) \right\}.
\]
\end{itemize}
\end{theorem}

\begin{proof} $(i)$ Fix $s\in \R$ and $q\in [1,\infty]$. 
By Propositions \ref{main 2} and \ref{bd multip}, the operators $a(D)_{|\cS(X)}$ and $\hat c(D)_{|\cS(X)}$ have extensions to  operators $\cT_\Psi = \cT_{\Psi, a}$ and $\cT_{\Psi, \hat c}$ in  $\cL(B^{s,q}_\Psi(X),B^{s,q}_\Psi(Y))$ for every Banach function space $\Psi$ such that $M$ is bounded on $\Psi$ and its dual $\Psi'$. 
Since for every $w\in A_2$, $w^{-1}\in A_2$ and $(L^2_w)'=L^2_{w^{-1}}$, we can take $\Psi = L^2_w$ with an arbitrary $w\in A_2$.

Let  $w\in A_2$ and set $\Psi:=L^2_w$. For $f\in \cS(X) \subset B^{s,q}_\Psi(X)$ one can easily check that $\cT_\Psi f = a(D)f\in B^{s,q}_\Psi(Y)$ is a solution of \eqref{P}.  
Recall that $\cA, \cB, \cP \in \cL(\cS'(Y), \cS'(X))$, $\partial \in \cL(\cS'(X))\cap \cL(\cS'(Y))$ and $\cT_{\Psi, \hat c} \in \cL(B^{s,q}_\Psi(Y),B^{s,q}_\Psi(X))$. 
By Lemma \ref{completness}, $B^{s,q}_\Psi(X) \hookrightarrow \cS'(X)$. If $q<\infty$, then $\cS(X)$ is a dense subset of $B^{s,q}_\Psi(X)$.  In the other cases, $\cS(X)$ is $\sigma(B^{s,\infty}_\Psi(X),B^{-s,1}_{\Psi'}(X^*)$-dense in $B^{s,q}_\Psi(X)$, 
 the operator $\cT_\Psi$ is  $\sigma(B^{s,\infty}_\Phi(X), B^{-s,1}_{\Phi'}(X^*))$-to-$\sigma(B^{s,\infty}_\Phi(Y), B^{-s,1}_{\Phi'}(Y^*))$-continuous, and the similar continuity property holds for $\cT_{\Psi, \hat c}$\,; see Proposition \ref{bd multip}.  
 Hence, in any case, $\cT_\Psi f$ is a solution of \eqref{P} for every $f\in B^{s,q}_\Psi(X)$.  

Now, take an arbitrary $\Phi$. Let $\cT_\Phi$ be the extension of $a(D)_{|\cS(X)}$ to an operator in $\cL(B^{s,q}_\Phi(X), B^{s,q}_\Phi(Y))$. 
Let $f\in B^{s,q}_\Phi(X) \cap \Phi(X)$. By Theorem \ref{RdeF}(i), there exists $w\in A_2$ such that $f\in L^2_w(X)=:\Psi(X)$. Moreover, note that $f_N:= \psi(2^{-N}D)f \in B^{s,q}_\Phi(X) \cap B^{s,q}_\Psi(X)$ for every $N\in \N$, $f_N \rightarrow f$ as $N\rightarrow \infty$ in  $B^{s,q}_\Phi(X)$ if $q<\infty$ and in the $\sigma(B^{s,\infty}_\Phi(X),B^{-s,1}_{\Phi'}(X^*))$-topology of $B^{s,\infty}_\Phi(X)$ in the other case; see Lemma \ref{completness}.  By Proposition \ref{bd multip}, $\cT_\Phi f_N = \cT_\Psi f_N$ and 
\[
\partial \cP \partial \cT_\Phi f   + \cB \partial \cT_\Phi f + \cA \cT_\Phi f + \cT_{\Phi, \hat c} \cT_\Phi f = f \quad \quad \textrm{ in } \cS'(X).
\] 
Since  $B^{s,q}_\Phi(X) \cap \Phi(X)$ is dense in $B_\Phi^{s,q}(X)$ if $q<\infty$, and $\sigma(B^{s,\infty}_\Phi(X),B^{-s,1}_{\Phi'}(X^*)$-dense in $B^{s,\infty}_\Psi(X)$ if $q=\infty$, the continuity argument based on Proposition \ref{bd multip} and Remark \ref{norms}(b) shows that the above equation holds for every $f\in B^{s,q}_\Phi(X)$.

By our assumptions and Proposition \ref{main 2} the functions 
$a_0, a_1, Aa(\cdot)$, and   $\hat c(\cdot)a(\cdot)$ belong to $\widetilde{\m}_2(\cL(X))$. Therefore, we can apply  Proposition \ref{bd multip} to these functions. One can easily check, that the corresponding operators in $\cL(B^{s,q}_\Phi(X))$, postulated in Proposition \ref{bd multip}, are given by 
\[
\partial \cP \partial \cT_\Phi, \quad \cB \partial \cT_\Phi, \quad \cA \cT_\Phi, \quad  T_{\Phi, \hat c} T_\Phi, 
\] 
respectively. Recall that $\cT_\Phi=\cT_{\Phi, a}$ corresponds to the function $a$. 
It proves that \eqref{P} has $B^{s,q}_\Phi$-maximal regularity property. 

Now we make a closer analysis to derive strong solutions of \eqref{P}.
Fix $f\in B^{s,q}_\Phi(X)$ and put $u:= \cT_\Phi f  \in  B^{s,q}_\Phi(Y)$. 
Since 
\[
 \partial (\partial \cP u  + \cB u) = \partial \cP \partial u + \cB \partial u \in B^{s,q}_\Phi(X)
\] 
and $v, \partial v \in B^{s,q}_{\Phi}(X)$ if and only if $v \in  B^{s+1,q}_{\Phi}(X)$ (see Lemma \ref{lem equiv norms}), we get $\partial \cP u  + \cB u \in B^{s+1,q}_\Phi(X)$, $\cB u\in B^{s+1,q}_{\Phi}(X)$, and, consequently, $\cP u \in  B^{s+2,q}_{\Phi}(X)$. 

For $s>0$ we have that $B^{s,q}_{\Phi}(X)\subset \Phi(X)\subset L^1_{loc}(X)$; see Lemma \ref{completness}. Thus,  $Bu\in W^{1,1}_{loc}(X)$ and $Pu\in W^{1,2}_{loc}(X)$, $B\partial u = (Bu)'$, $\partial(\cP\partial u) = (Pu)''$, $Au$, $\cT_{\Phi, \hat c} u$ belong to $B^{s,q}_{\Phi}(X)$,  and \eqref{P} can be interpreted as
\[
(Pu)''(t) + (Bu)'(t) +Au(t) +  c \ast u(t) = f(t) \quad \textrm{  for a.e. } t \in  \R. 
\]
In addition, if we know that $u \in W^{1,1}_{loc}(X)$, then by Lemma \ref{strong solution}, $\cB \partial u = Bu'$ and $\partial \cP \partial u = (P u')'$ and \eqref{P} takes the form \eqref{P strong}.
One can easily check that our additional assumption that the function $(\cdot) a(\cdot)$ is in  $\widetilde{\m}_2(\cL(X))$ yields $u \in W^{1,1}_{loc}(X)$. It completes the proof of $(i)$.\\

$(ii)$ Since the inversion map is analytic, $a \in \cC^{3}(\dot \R; \cL(X,Y))$, where $a(s):=b(s)^{-1}$ for $s\neq 0$. 
Note that for every $f\in \cF^{-1}\cD_0(X)$ the function $u:= \cF^{-1}(a(\cdot) \cF f) \in L^p(Y)$ satisfies \eqref{P}. Therefore, the operator $a(D)_{|\cF^{-1}\cD_0(X)}$ has an extension to an operator $T$ in $\cL(L^p(X),L^p(Y))$, which assigns to each $f\in L^p(X)$ the corresponding solution $u\in L^p(Y)$ of \eqref{P}. A standard argument shows that $a\in L^\infty(\cL(X,Y))$. Similarly, we get that the functions $a_0$ and $a_1$ given by \eqref{functions} belong to $L^\infty(\cL(X))$. Consequently, by Proposition \ref{main 2}, $a, a_0, a_1$ satisfy the $\m_3$-condition provided that $\hat c$ does. 
Now, Proposition \ref{bd multip}(ii) shows that the multipliers $a(D)$, $a_0(D)$, $a_1(D)$, $\hat c(D)$ have  extensions to operators $\cT_\E :=\cT_{\E,a}\in \cL(\E(X), \E(Y))$, $\cT_{\E,a_0}, \cT_{\E,a_1} \in \cL(\E(X))$, and $\cT_{\E,\hat c} \in \cL(\E(Y),\E(X))$, respectively.  Moreover, we have that 
\[
\cT_{\E,a_0} = \cB \partial \cT_{\E} , \quad  \cT_{\E,a_1} = \partial \cP \partial \cT_{\E},\quad  \cT_{\E, \hat c a} = \cT_{\E, \hat c} \cT_{\E}. 
\]

By similar arguments to those presented in the proof of the part $(i)$, first one can show that \eqref{P} has $\Phi$-maximal regularity. If $\E \neq \Phi$, then relying on the fact that $\E(X)\cap \Phi(X)$ is a dense subset of $\E(X)$ (see Lemma \ref{completness}) and the fact that $\cT_{\E}$ is consistent with $\cT_{\Phi}$ (see Theorem \ref{ext general}(ii)), we get that for every $f\in \E(X)$ the distribution $u:= \cT_{\E, a} f$ is a solution of \eqref{P} in $\E(X) \subset \cS'(X)$ with desired properties, i.e. $\partial \cP \partial u,\, \cB \partial u,\,  \cA u,\, c\ast u \in \E(X)$.

For the second statement, if $(\cdot)a(\cdot) \in L^\infty(\cL(X))$, then by Propositions \ref{main 2} and \ref{bd multip}, one can argue that 
$ \cT_{\E, (\cdot)a(\cdot)}f =  \partial \cT_{\E} f \in \E(X)$ for every $f\in \E(X)$. 
The closedness of $B$ and $P$ implies that 
\[
\cB \partial \cT_\E,\, \cP \partial \cT_E \in \E(X).
\]
If, additionally, $\E\subset L^1_{loc}$, then again by Lemma \ref{strong solution}, we get that the function $u = \cT_{\E} f$ is a strong solution of \eqref{P}, $\partial u = u'$ and  $ \partial \cP \partial u = (Pu')' \in \E(X)$. This completes the proof. 
\end{proof}

Under additional assumptions on the geometry of the underlying Banach space $X$ one can relax the regularity conditions imposed on the function $\hat c$. We refer to Remark \ref{R bound} for the notion of the $R\m_1$-condition, which is involved in the formulation of the following result.

\begin{corollary}\label{for umd}
Let  $A, B$ and $P$ be closed, linear operator on a Banach space $X$ with  $U\!M\!D$ property such that $D_A \hookrightarrow D_B, D_P$. Let $c:=\cF^{-1} \hat c$ for some function $\hat c \in L^\infty(\cL(Y,X))$. 

Assume that for every $t\neq 0$
\[
b(t)= - t^2 P + it B + A + \hat c(t) \in \cL(Y,X)
\] is invertible and let $a$ and $a_0,a_1$ be given as in \eqref{functions}. 
Assume that the functions $a$, $a_0$, $a_1$ are $R$-bounded and that $\hat c$ satisfies the $R\m_1$-condition.

Then, the conclusion of Theorem \ref{main for itegro-dff}(ii) holds for the corresponding problem \eqref{P}.  
\end{corollary}
\begin{proof}
By the means of \cite[Theorem 3.5]{FaHyLi20}(see Remark \ref{R bound}), Proposition \ref{main 2},  and Theorem \ref{ext general}, the proof follows the lines of the proof of Theorem \ref{main for itegro-dff}(ii) above.  
\end{proof}

\begin{remark}
The $L^p$- and $B^{s,q}_p$-maximal regularity of the problem \eqref{P} with a particular form of the convolution term $c\ast u$ was studied in \cite{ApKe20} and \cite{KeAp20}, respectively. Corollary \ref{for umd} provides an extension of \cite[Theorem 4.1]{ApKe20}. 
In these papers the reader can find a list of concrete models arising in many areas of applied mathematics, which are cover by this abstract problem. In particular, following presentation in \cite[Section 5]{KeAp20}, one can apply Theorem \ref{main for itegro-dff} to several concrete problems arising in physics. We left such adaptation of \cite[Section 5]{KeAp20} for the interested reader.  
\end{remark}

\section{Particular forms of the equation \eqref{P} }\label{cases}

We do not attempt to give a systematic survey on particular cases of \eqref{P} and address only some issues directly related to the assumptions of Theorem \ref{main for itegro-dff}. As in the previous section, we assume that $A$, $B$, and $P$ are closed linear operators on a Banach space $X$ such that $D_A\hookrightarrow D_B,D_P$. 

\subsection{The evolutionary differential equations} Consider the equation \eqref{P} with $c=0$, i.e.
\begin{equation}\label{P no c}
\partial \cP \partial u + \cB \partial u + \cA u = f \quad \textrm{ in } \cS'(X).
\end{equation}
Such abstract evolution equation reduces to further ones, which have been intensively studied in the literature. We mention here a few of them. 

When $P=0$ and $B=I$, \eqref{P no c} takes the form of the problem \eqref{P1} already considered in Section 1. Its variant on the half-line $\R_+:=[0, \infty)$, i.e. the following first-order Cauchy problem 
\begin{equation}\label{ACP}
 u' + A u = f  \textrm{ on } \R_+ \quad \textrm{ with }  u(0)= 0
\end{equation}
is a model one for the study of the $L^p$-maximal regularity of \eqref{P no c}.  
We refer the reader to \cite{KuWe04} or \cite{DeHiPr03a} for the background on the $L^p$-maximal regularity the abstract Cauchy problem \eqref{ACP}. 
 In this context, we point out that the kernels of Fourier multipliers involved in the study of maximal regularity have supports in $\R_+$, which allows considering a larger class of Banach function spaces $\Phi$ over $(\R_+, \ud t)$, and consequently proving the existence of strong solutions of \eqref{ACP} for a larger class of functions $f$; see \cite[Section 5]{ChKr17} and the references therein.

The problem \eqref{P1} on the line was considered by Mielke \cite{Mi87}, who characterized its $L^p$-maximal regularity in the class of closed, linear operators $A$ on Hilbert spaces $X$. 
His result was later extended for $X$ with $U\!M\!D$ property by Arendt and Dueli \cite[Theorem 2.4]{ArDu06}. It provides a direct counterpart of Weis' characterisation of the $L^p$-maximal regularity for the Cauchy problem \eqref{ACP}; see \cite{We01}. 
More recently, a similar strategy was adapted, e.g. in \cite{ApKe20}(cf. also \cite{PoPo17}), to give a counterpart of Weis' result for \eqref{P} with a special form of $c$. Note that the main point is to give a variant of Mielke's argument from \cite[Lemma 2.3]{Mi87} for \eqref{P} (cf. also a different argument given by Dore \cite{Do93} for  \eqref{ACP}). More precisely, one need to argue that if \eqref{P} has $L^p$-maximal regularity then there exists $\alpha>0$ such that the problem 
\begin{equation}\label{Mielke 1}
(Pv')' + Bv' + A v  = g + \alpha Pv - \alpha \sgn(\cdot) Bv - e^{-\alpha|\cdot|} c\ast (e^{\alpha|\cdot|}v)
\end{equation}
is well-posed in $L^p(X)$, i.e. for all $g\in L^p(X)$ there exists a strong solution $v$ in $L^p(Y)$. Note that the well-posedness of \eqref{Mielke 2} in $L^p(X)$ is equivalent to that of \eqref{P} in $L^p(\R,e^{-\alpha|\cdot|} \ud t;X)$, which further yields to the invertibility of the operator $b(t)$ in \eqref{fun c} for all $t\in \R$. 
If $\cT$ denotes the solution operator for \eqref{P}, which exists by the assumed $L^p$-maximal regularity of \eqref{P}, then  \eqref{Mielke 1} yields 
\begin{equation}\label{Mielke 2}
v - \cT\left( \alpha [P - \sgn(\cdot) B]v +  e^{-\alpha|\cdot|} c\ast (e^{\alpha|\cdot|}v) - c\ast v \right) = \cT g. 
\end{equation}
Therefore, to show that that the linear operator on the left side of \eqref{Mielke 2} is invertible on $L^p(X)$ for some $\alpha>0$, and make all considerations rigorous, one need to make some further restriction on the convolutor $c$ (cf. \cite{ApKe20} and \cite{PoPo17}). Below we chose the simplest case, that is, when $c=0$. Then, $\|\alpha \cT(P - \sgn(\cdot) B)\|_{\cL(L^p(X))}<1$ for all $\alpha$ small enough and $v = (I - \alpha \cT(P - \sgn(\cdot) B))^{-1} \cT g$, i.e. \eqref{Mielke 1} is well-posed in $L^p(X)$.

The following result shows that the assumption made in Theorem \ref{main for itegro-dff}(ii) on the invertibility of $b(s)$ with $c=0$ is a consequence of the $L^p$-maximal regularity of \eqref{P no c}. 

\begin{proposition}\label{OM class} Let  $A, B$ and  $P$ be closed, linear operator on a Banach space $X$ with $U\!M\!D$ property such that $D_A \hookrightarrow D_B, D_P$. 

 If \eqref{P no c} has $L^p$-maximal regularity, then 
for every $t\in \R$ the operator 
\[
b(t):= -t^2P + itB + A \in  \cL(Y, X)
\] 
is invertible. 

Furthermore, the functions $a,a_0,a_1$  (given by \eqref{functions}) are bounded, belong to the $\cO_M$-class, and satisfy the $\widetilde\m_\gamma$-condition for all $\gamma\in \N$. In particular, the conclusion of Theorem \ref{main for itegro-dff}(ii) holds for \eqref{P no c}. 
\end{proposition}
The first statement is a special case of \cite[Proposition 3.11]{ApKe20}. 
The second one follows from standard argument and Proposition \ref{main 2}.

\subsection{The special form of the convolution term} 
For an potential application of Theorem \ref{main for itegro-dff}, it is natural to specify the form of the convolution term in \eqref{P} as follows:  
\begin{equation}\label{c decomp}
c= \partial c_0 + c_1 \quad \textrm{ with } c_j \in \cS'(\cL(Y_j ,X))  
\end{equation}
and further one can assume that 
\[
c_i=  \delta \otimes C_{0,i} + c_{1,i} \textrm{ with } C_{i,0}\in \cL(Y_i, X) \textrm{ and } c_{i,1} \in \cS'(\cL(Y_i,X)).
\]
Then, 
\[
c\ast u = C_{0,0} \partial u +c_{0,1} \ast u + C_{1,0} u +  c_{1,1} \ast u.
\]
When we further take $C_{0,0} :=\alpha_0 I$, $C_{1,0} := \beta_0 C$ for $\alpha_0, \beta_0  \in \C$, $C \in \cL(Y_1, X)$, and $c_{0,1}:=\alpha_1 I$ and $c_{1,1}:= \beta_1 C$ for $\alpha_1, \beta_1 \in S'$, then the convolution term $c\ast u$ corresponds to the one from the heat-conduction problem with memory \cite[Section 5]{Pr93}. 
Such form of the convolutor $c$ was considered in the context of \eqref{P} in \cite{ApKe20}. In this {\it scalar} case one can easily specify the assumptions on the functions $\hat \alpha_1$, $\hat \beta_1$ and the operator $C$,  which imply that $c$ satisfies the assumptions of Theorem \ref{main for itegro-dff}; cf. \cite{ApKe20} and references therein.

\subsection{The abstract convolution equation} Finally,  we comment on \eqref{P} with $A=B=P=0$, that is, we deal only with the convolution equation:
\begin{equation}\label{conv eq}
c\ast u = f \quad \textrm{ in } S'(X).
\end{equation}
Assume that $c$ is of the form \eqref{c decomp}, $Y_0 := X$, $Y_1 := Y \hookrightarrow X$, and $\hat c_i := \cF c_i \in L^1_{loc}(\cL(Y_i, X))$ $(i=0,1)$. Then, 
$\hat c  := \cF c = i(\cdot) \hat c_0 + \hat c_1 \in L^1_{loc}(\cL(Y, X)) $

We specify the assumptions on $c_i$ to get the solvability of \eqref{conv eq}.

\begin{proposition}\label{conv eq est} 
Suppose that $\hat c_i \in \widetilde\m_2(\cL(Y_i,X))$ and for every $s\in \dot \R$
\[ 
\hat c(t) =  it \hat c_0 (t)  + \hat c_1(t)\in \cL(Y,X)
\] is invertible. Set $a(t):= \hat c(t)^{-1}$ and $d(t):=t a(t) \in \cL(X)$,  $t\in \dot \R$. 

If $a\in L^\infty(\cL(X,Y))$ and $d \in L^\infty(\cL(X))$, then 
$a\in \widetilde\m_2((\cL(X,Y))$ and $d \in \widetilde\m_2(\cL(X))$.

Moreover, for every $s\in \R$, $q\in [1,\infty]$, and every Banach function space $\Phi$ such that the Hardy-Littlewood operator $M$ is bounded on $\Phi$ and $\Phi'$, 
the operator $a(D)$ extends to an isomorphism from $B^{s,q}_\Phi(X)$ onto $B^{s,q}_\Phi(Y)\cap B^{s+1,q}_\Phi(X)$. In particular, for every $f\in B^{s,q}_\Phi(X)$, the problem \eqref{conv eq} has a unique solution $u \in B^{s,q}_\Phi(Y)\cap B^{s+1,q}_\Phi(X)$.  
\end{proposition}

\begin{proof} The proof reproduces the arguments already presented in the proofs of Proposition \ref{main 2} and Theorem \ref{main for itegro-dff}. We present only its line and leave the details for the reader. 

First, note that in the context of Proposition \ref{main 2}, here $b = \hat c = i(\cdot ) c_0 + \hat c_1$ and, by Leibniz' rule, for every $\alpha = 1,2,3$ we get 
\[
b^{(\alpha)}(t) = it \hat c_0^{(\alpha)}(t) + \alpha i c_0^{(\alpha -1)}(t) + \hat c_1^{(\alpha)}(t) \quad (t\in \dot \R).
\] 
Therefore, combining \eqref{est 1}, \eqref{est 2}, and \eqref{est 3} with our assumptions on $a$, $c$,  and $d$ we easily get that $a$ and $d$ satisfy the $\widetilde \m_2$-condition.

Now, following the arguments presented in the proof of Theorem \ref{main for itegro-dff}(i) we get that the multipliers $a(D)_|$ and $d(D)_|$ extend  to operators $\cT_a \in \cL(B^{s,q}_\Phi(X), B^{s,q}_\Phi(Y))$ and $\cT_d = \partial \cT_a \in \cL(B^{s,q}_\Phi(X))$, respectively. Similarly, the multipliers $\hat c_i(D)_|$ extend to operators $\cT_{\hat c_i} \in \cL(B^{s,q}_\Phi(Y_i), B^{s,q}_\Phi(X))$ $(i=0,1)$.

Moreover, by a similar density argument as the one applied in the proof of Theorem \ref{main for itegro-dff}, we have that for every $f\in B^{s,q}_\Phi(X)$, the distribution $u:= \cT_a f$ satisfies \eqref{conv eq}, where $c\ast u$ is defined as  
$\cT_{\hat c_0} \partial u + \cT_{\hat c_1} u$. 
To show that $\cT_a$ is an isomorphism in $\cL(B^{s,q}_\Phi(X), B^{s,q}_\Phi(Y)\cap B^{s+1,q}_\Phi(X))$, first note that by Lemma \ref{lem equiv norms}, $\cT_af \in B^{s+1,q}_\Phi(X)$ for every $f\in B^{s,q}_\Phi(X)$, and that $\cT_a$ is onto, since each $u\in B^{s,q}_\Phi(Y)\cap B^{s+1,q}_\Phi(X)$ is a solution of \eqref{conv eq} for $f :=\cT_{\hat c_0} \partial u  + \cT_{\hat c_1} u \in B^{s,q}_\Phi(X)$. 
As we already noted the operator $\cT_{\hat c_0} \partial  + \cT_{\hat c_1}$ is the left inverse to $\cT_a$. It completes the proof.

\end{proof}
The fact that we consider the Besov spaces with negative order $s$ allows us to derive a regularity property of solutions of \eqref{conv eq} with an arbitrary $f\in \Phi(X)$, which  is  similar to that one given in \cite[Corollary 8.3]{Am97}. Note that the $\Phi$-maximal regularity of \eqref{conv eq} can be interpreted as a limit case of such result.

\providecommand{\bysame}{\leavevmode\hbox to3em{\hrulefill}\thinspace}

\bibliographystyle{amsplain}

\end{document}